\documentclass[preprint,10pt]{amsart}%
\usepackage[margin=1.0in]{geometry}
\usepackage{amsmath}
\usepackage{amssymb}
\usepackage{amsfonts}
\usepackage{upref,amsthm,amsxtra,exscale}
\usepackage{cite}
\usepackage[colorlinks=true,urlcolor=blue,
citecolor=red,linkcolor=blue,linktocpage,pdfpagelabels,
bookmarksnumbered,bookmarksopen]{hyperref}
\usepackage{epsfig,graphics,color}
\usepackage{graphicx}%
\newtheorem{theorem}{Theorem}[section]
\theoremstyle{plain}

\newtheorem{lemma}[theorem]{Lemma}

\newtheorem{Remark}{Remark}

\numberwithin{equation}{section}

\begin{document}
\title[An Abstract Linking Theorem Applied to Indefinite Problems ]{An Abstract Linking Theorem Applied to Indefinite Problems Via Spectral Properties}
\author{Liliane A. Maia}
\address{Department of Mathematics, UNB, 70910-900 Brasilia, Brazil.}
\email{lilimaia@unb.br}
\author{ Mayra Soares}
\address{Department of Mathematics, UNB, 70910-900 Brasilia, Brazil.}
\email{ssc\_mayra@hotmail.com}
\thanks{Research supported by FAPDF 0193.001300/2016, CNPq/PQ 308378/2017-2 (Brazil)}
\date{\today}

\begin{abstract}
An abstract linking result for Cerami sequences is proved without the Cerami condition. It is applied directly in order to prove the existence of critical points for a class of indefinite problems in infinite dimensional Hilbert Spaces. The applications are given to Schr\"odinger equations. Here spectral properties inherited by the potential features are exploited in order to establish a linking structure, hence hypotheses of monotonicity on the nonlinearities are discarded.
\end{abstract}
\maketitle

\section{Introduction}
\label{sec:introduction}

\qquad In this work the groundbreaking paper \cite{BR} by V. Benci and P.H. Rabinowitz is revisited.
The aim is to prove an abstract linking theorem for Cerami sequences \cite{C}, which will complement related works found in the literature and make possible to extend for many applications.
Our interest in applications are twofold, on one hand, extending results for existence of solutions to nonlinear Schr\"odinger Equations, Elliptic Systems or even Hamiltonian Systems, with  very general potentials which make the problems strongly indefinite. On the other hand working with nonlinear terms which do not satisfy any monotonicity condition such as those required to perform projections on the so called Nehari Manifold as in \cite{Pa,SW}, for instance. 
In this purpose spectral properties of self-adjoint operators are going to be exploited, in order to get the geometry of a linking structure and then apply an abstract result to obtain a critical point to the functional associated to the nonlinear equation, namely a solution to the problem. Furthermore, a compactness structure given by Cerami sequences, $(C)_c$ sequence for short, is faced here, since asymptotically linear problems at infinity are studied. Thus, inspired by the theory developed in \cite{BR}, in this work a more general version of their main result, Theorem 1.29, is provided for $(C)_c$ sequences. To do so, a Deformation Lemma adapted for Cerami sequences is proved and then the abstract results obtained by V. Benci and P. Rabinowitz are extended.

\qquad Our main result, developed throughout this paper is the following:\\

\hspace{-0.4cm}\textbf{Linking Theorem for Cerami Sequences:} Let $E$ be a real Hilbert space, with inner product $\big( \cdot, \cdot \big)$, $E_1$ a closed subspace of $E$ and $E_2=E_1^{\perp}$. Let $I \in C^1(E,\mathbb{R})$ satisfying:\\

\hspace{-0.5cm}$(I_1) \ \ I(u)=\dfrac{1}{2}\big(Lu,u\big)+B(u),$ for all $u \in E$, where $u=u_1 +u_2 \in E_1 \oplus E_2$, $Lu=L_1u_1 + L_2u_2$ and\  $L_i: E_i \rightarrow E_i, \ i=1,2$ is a bounded linear self adjoint mapping.

\hspace{-0.5cm}$(I_2) \ \ B$ is weakly continuous and uniformly differentiable on bounded subsets of $E$.

\hspace{-0.5cm}$(I_3) \ $ There exist Hilbert manifolds $S,Q \subset E$, such that $Q$ is bounded and has boundary $\partial Q$, constants $\alpha > \omega$ and $v \in E_2$ such that\\
$(i) \ S\subset v + E_1$ and $I \geq \alpha$ on $S$;
$(ii) \ I\leq \omega$ on $\partial Q$;
$(iii) \ S$ and $\partial Q$ link.

\hspace{-0.5cm}$(I_4) \  \ $
If for a sequence $(u_n)$, $I(u_n)$ is bounded and  $\left(1+||u_n||\right)||I'(u_n)|| \to 0$, as $n\to +\infty$, then $(u_n)$ is bounded.\\

\hspace{-0.5cm}Then I possesses a critical value $c\geq \alpha$.\\

\qquad It is important to highlight $(I_2)$ implies that $B'(u)$ maps weakly convergent to strongly convergent sequences, which gives a kind of partial compactness for $I$. Moreover, $(I_4)$ is a weakened version of Cerami condition, so denoted $(C)_c$ condition, once the boundedness of any $(C)_c$ sequence will be enough to look for a nontrivial critical point, without wondering whether it has a convergent subsequence (see the first paragraph of the proof of Theorem \ref{ALT}).
Together, hypotheses $(I_2)$ and $(I_4)$ ensure the existence of a critical value $c$ which can be characterized as a minimax level. Furthermore, hypotheses $(I_1)$ and $(I_3)$ produce a quite general linking geometry for functional $I$, so that both subspaces in the Hilbert's decomposition are allowed to be infinite dimensional. The conjunction of these hypotheses reproduces sufficient tools to obtain a nontrivial critical point for $I$ in the desired applications scenario.
In fact, our approach of constructing a linking structure, by means of a sharp study on the spectral properties of the Schr\"odinger
operator is a methodological novelty. This idea extremely enhances the current references since it
lights up the fundamental relation between the asymptotic behavior of the nonlinear term and the spectrum features. Therefore, the purpose of revisiting the core of linking structures was precisely to understand this interaction and
discard any unnecessary hypothesis.

\qquad The pioneering work in this direction is \cite{BBF} by P. Bartolo, V. Benci and D. Fortunato, where a Deformation Lemma for Cerami sequences was developed
assuming $(C)_c$ condition, as a qualitative deformation lemma, with the purpose of extending previous critical point results to non super-quadratic problems. Thereafter, D. Costa and C. Magalhães  in  \cite{CM} proved abstract linking results for strongly indefinite non-quadratic problems on bounded domains, making use of the Deformation Lemma introduced in \cite{BBF} and proving that under their assumptions the associated functional satisfied $(C)_c$ condition. Alternatively, as aforementioned, the same lines as \cite{BR} are followed, hence a deformation lemma without $(C)_c$ condition is proved, and furthermore, our version of linking theorem only requires the boundedness of Cerami sequences.

\qquad In the literature one also finds a paper by G. Li and C. Wang \cite{LW}, which presented a similar argument, introducing a new kind of deformation lemma, without $(C)_c$ condition, but subsequently used in a linking theorem under $(C)_c$ condition. Moreover, in the abstract result they required that one of the subspaces in the linking decomposition being finite dimension, while in our result both subspaces in the decomposition may be of infinite dimension. Their construction was inspired by ideas of M. Willem \cite{W}, for the quantitative deformation lemma. 
Although a kind of deformation lemma is developed, it is deeply different from theirs, since nonstandard ideas in \cite{BR} are closely followed. 
In fact, the mapping $\eta$ in deformation is in general determined by solving an appropriate  initial value problem involving $I'(\eta)$. However, this is not suitable for our purposes, because is necessary to construct an $\eta$ satisfying special properties, which will be fundamental in attaining the critical minimax level. 

\qquad It is also important to mention the theory developed by W. Kryszewski and A. Szulkin in \cite{KS}, where they solved a more general class of superlinear problems, with assumptions of periodicity. Developing a new degree theory and a weaker topology, they generalized abstract linking theorems introduced in \cite{BR} also working with Palais-Smale sequences. Following the same idea, in \cite{LS} G. Li and A. Szulkin extended the results in \cite{KS} obtaining a $(C)_c$ sequence for the asymptotically linear case. Nevertheless, so as to get a non-trivial solution, without $(C)_c$ condition, these authors required
extra assumptions, including a monotonicity condition on the nonlinearity in an auxiliary problem solved in \cite{SzZo}, 
which had been treated by adapting techniques in \cite{KS} and \cite{J}. 

\qquad Posteriorly, T. Bartsch and Y. Ding \cite{BD} complemented the results in \cite{KS} considering both, Palais-Smale and Cerami sequences, in order to apply their abstract results to a Dirac equation where the nonlinearity could be asymptotically linear or superlinear at infinity. Similarly, in \cite{DR}, Y. Ding and B. Ruf worked with an asymptotically linear problem with a Dirac operator, but without periodicity conditions. Their operator satisfies that the essential spectrum is $\mathbb{R} \setminus (-a,a)$ and that the discrete spectrum intersects the interval $(0,q_0)$, for some positive $q_0$. Then they could make use of discrete and positive eigenvalues in the linking structure and apply a particular case of the result in \cite{BD}, so as to obtain a Cerami sequence. Under their assumptions they were able to prove that their functional satisfied $(C)_c$ condition, which yielded their results. 

\qquad It's worth to highlight that \cite{DR} adapted assumptions and arguments introduced in another very inspiring paper \cite{DJ}, where L. Jeanjean and Y. Ding worked with Hamiltonian Systems, looking for homoclinic orbits, without any periodicity condition. These authors also applied the abstract critical point theory developed in \cite{BD}, and in order to recover the desired compactness they imposed hypotheses controlling the size of the nonlinearity with respect to the behavior of the potential at infinity. Thus, their assumptions yielded the  linking geometry, and provided $(C)_c$ condition. In contrast to the argument presented by these authors, our approach does not require the guarantee of $(C)_c$ condition, the necessary compactness to solve the problem is embedded inside the four conditions assumed in the abstract result.

\qquad Still referring to abstract results involving linking structure, it is as well known that M. Schechter and W. Zou have developed many relevant papers in this spirit, see especially \cite{Sc,ScB,ScZ2,ScZ1} among other works by these authors. In our understanding, their results are away from ours in the sense that, roughly speaking, they usually work with weaker linking geometries in order to get either a Palais-Smale or a Cerami bounded sequence. Then they apply widely alternative arguments to find a solution to the proposed application. On the other side, our idea is to obtain a result which could ensure the existence of a nontrivial critical point directly, without stressing either on geometry or on compactness of the associated functional, separately. Notwithstanding, it's worth pointing out clever abstract results obtained in \cite{ScZ2,ScZ1} (cf. Theorem 2.1 in both), where the authors also made use of ``Monotonicity Trick" developed by L. Jeanjean in \cite{J}, for the purpose of getting critical points for a family of functionals, converging to a critical point the functional associated to the initial problem. These results have been applied to solve asymptotically linear problems with spectral properties similar to those presented in this paper, see \cite{CZ}, for instance.

\qquad Here we present two applications in Schr\"odinger equations for our abstract result. Other applications can be found in \cite{So}, where M. Soares proves the existence of solution to Hamiltonian and Elliptic systems by applying this abstract result. 

\qquad For our applications we consider problem $(P)$
\begin{equation}\label{P}
-\Delta u + V(x)u = g(x,u) \text{ \ in \ } \mathbb{R}^N, \ \ N\geq3, 
\end{equation}
in the case where $g(x,s) = h(x)f(s),$ and $h$ satisfies\\

\hspace{-0.5cm}$(h_1) \ h \in L^{\infty}(\mathbb{R}^N)\cap L^{q}(\mathbb{R}^N), \ q = \frac{2^*}{2^*-p}$,  for some $p \in (2,2^*)$ and $h >0$ almost everywhere.\\

\qquad Furthermore, $f$ is asymptotically linear satisfying\\

\hspace{-0.5cm}$(f_1) \ \ f \in C(\mathbb{R},\mathbb{R})$ and $\displaystyle\lim_{s \to 0}\dfrac{f(s)}{s}=0;$

\hspace{-0.5cm}$(f_2)$ \ There exists $a>0$ such that $\displaystyle\lim_{s \to +\infty}\dfrac{F(s)}{s^2}=\dfrac{a}{2},$ where $F(s):= \displaystyle\int_0^sf(t)dt,$ and $F(s)\geq0$.

\hspace{-0.5cm}$(f_3)$ \ Setting $Q(s):=\frac{1}{2}f(s)s - F(s)>0$ for all $s \in \mathbb{R}/\{0\}$,
\[
\displaystyle\lim_{s\to +\infty} Q(s) = +\infty.
\]

\hspace{-0.5cm}Moreover, for the first application, we assume that $V$ satisfies:\\

\hspace{-0.5cm}$(V_1) \ \ V \in C(\mathbb{R}^N,\mathbb{R})$ and $\displaystyle\lim_{|x|\to +\infty}V(x)=V_{\infty}>0;$

\hspace{-0.5cm}$(V_2)$ \ Setting $A := -\Delta + V(x)$, as an operator of $L^2(\mathbb{R}^N)$, and denoting by $\sigma(A)$ the spectrum of $A$,
\[
\sup\left[\sigma(A)\cap(-\infty,0)\right]= \sigma^{-} < 0 < \sigma^{+} = \inf\left[\sigma(A)\cap(0,+\infty)\right].
\]
\medskip

\qquad This application was inspired by L. Maia, J. Oliveira Junior and R. Ruviaro \cite{MOR}, where they solved problem $(P)$ with potential $V$ satisfying $(V_1)-(V_2)$. In addition, they required that $0 \notin \sigma(A)$ and assumed some more hypotheses of decay and compactness \ on $V$. \ About the \ nonlinearity \ they set \ $h\equiv 1$, \ $f \in C^3(\mathbb{R}^N, \mathbb{R})$, with some growth hypotheses on its derivatives, and assumed ${f(s)}/{s}$ being increasing as well.
Since this kind of potential ensures that the subspace where $A$ is negative definite is finite dimensional, they could apply the aforementioned version of Linking Theorem introduced by G. Li and C. Wang in \cite{LW} to get $(C)_c$ sequence.
They used the associated problem at infinity and a Splitting Lemma to compare the levels of both problems and get the necessary compactness. Trying to improve their result, our abstract linking theorem for $(C)_c$ sequences is applied and a nontrivial critical point is obtained straightway, avoiding such monotonicity assumptions on $f$. 

\qquad Staring at our hypotheses, it is also possible to say that our second application complements the work by L. Jeanjean and K. Tanaka in \cite{JT}. In fact, they assumed $V(x) \geq \alpha >0$, and so they worked with Ekeland's principle to get a $(C)_c$ sequence and due to the geometry of their functional, they applied the Mountain Pass Theorem to get a critical point. They also worked with an asymptotically linear problem where ${f(s)}/{s}$ is not necessarily increasing. In addition, they assumed $h\equiv 1$ and
${f(s)}/{s} \to a >\inf \sigma(A)> 0$ as $|s| \to + \infty$. Differently, in our case $V$ changes sign and $\inf \sigma(A)<0$, which implies a linking geometry and prevents us to use the same argument. However, considering the positive spectrum a similar hypothesis is assumed:
\[
a>\inf_{u_1\in E_1, u\not=0}\dfrac{\displaystyle\int_{\mathbb{R}^N}\left(|\nabla u_1(x)|^2+ V(x)u^2_1(x)\right)\, dx }{\displaystyle\int_{\mathbb{R}^N}h(x)u^2_1(x)\ dx}\geq \dfrac{1}{h_\infty}\inf[\sigma(A)\cap(0,+\infty)] = \dfrac{\sigma^+}{h_\infty}>0,
\]
where $||h||_{L^\infty(\mathbb{R}^N)} := h_\infty$ and $E_1$ is the subspace of $H^1(\mathbb{R}^N)$ on which operator $A$ is positive definite. This hypothesis allows to develop a linking structure.

\qquad On the other hand, the work \cite{CT} by D. Costa and H. Tehrani can be cited, since the same $V$ as theirs is presented here. More specifically, they required $(V_1)$, and our hypothesis $(V_2)$ implies that either $0$ is an isolated point of $\sigma(A)$, or it is in a gap of the spectrum, which is also required by them. However, they did not work with an asymptotically linear problem, but they assumed Ambrosetti-Rabinowitz well known condition, and required a nondecreasing nonlinearity. In their assumptions, $h=a$ is a sign-changing function in  $C^1(\mathbb{R}^N,\mathbb{R})$ and such that $\displaystyle\lim_{|x|\to +\infty}a(x)= a_{\infty}<0$, differently from $\displaystyle\lim_{|x|\to +\infty}h(x)= 0$, in our case. Moreover, instead of using an abstract linking theorem, they applied a method of approximations to solve their problem.

\qquad It is also worth to mention the paper \cite{ES} by A. Edelson and C. Stuart, since  assumptions close to theirs are considered here, however they asked ${f(s)}/{s}$ strictly increasing, which is removed here. Moreover, they applied the method of sub and super-solution and bifurcation to get a solution to their problem.

\qquad Finally, for the second application we keep all assumptions on $h$ and $f$, but on $V$ we assume $(V_2)$ and replace assumption $(V_1)$ by the following:\\

\hspace{-0.5cm}$(V_1') \ V \in C(\mathbb{R}^N,\mathbb{R})$ is $(2\pi)^N$-periodic in $x \in \mathbb{R}^N$.\\

\qquad This application is motivated by the fact that in virtue of $(V_1)$, the subspace in which operator $A$ is negative definite, is finite dimensional. Since this is irrelevant for applying our abstract result, we sought for a problem where both subspaces, in which operator $A$ is positive and negative  definite, are infinite dimensional. In fact, $(V_1')$ combined with $(V_2)$ ensure the desired. Although all other hypotheses are kept, this replacement changes completely the spectral properties of $A$, which are fundamental to determine the linking geometry.
 It would be interesting to note that, we only require $V$ being a periodic function in order to explore spectral properties, we do not need a periodic nonlinearity. Hence unlike most of the works in the literature (see \cite{ KS, BD}), we do not make use of periodicity to translate a $(C)_c$ sequence and ensure the existence of a critical point.

\section{The Notion of Linking and Some Definitions}

\qquad In this section the notion of link is presented and a new version of an abstract linking theorem is proved based on \cite{BR}, however with $(C)_c$ sequences. Throughout this section $E$ always denotes a Hilbert space, $E=E_1\oplus E_2$ and if $u\in E,$ write $u= u_1+u_2$ with $u_i \in E_i, \ i=1,2$, then set $P_iu:=u_i$, where $P_i:E\to E_i$ is the projector on $E_i, \ i=1,2$. Furthermore, mappings $h:[0,1]\times E\to E$ will be denoted by $h_t(u),$ and the closed ball of $E$ centered in zero with radius $r$, will be denoted by $B_r$. Furthermore, let $\mathcal{B}_{\tau}= (B_{\tau}\cap E_1)\oplus(B_{\tau}\cap E_2)$.

	Let	$\Sigma$ \  denote \ the \ class \ of \ mappings $\Phi \in C([0,1]\times E, E)$, for which $P_2\Phi_t(u)=u_2 - W_t(u)$, with $W_t$ compact for $t\in [0,1]$ and $\Phi_0(u)=u$. Let $S$ and $Q$ be Hilbert manifolds, $Q$ having a boundary, $\partial Q$, $S$ and $\partial Q$ ``link'' if whenever $\Phi \in \Sigma$ and $\Phi_t(\partial Q)\cap S= \emptyset$, for all $t\in [0,1]$, then $\Phi_t(Q)\cap S \not=\emptyset$, for all $t \in [0,1]$.
	

\begin{Remark}\label{s1r1}
	A geometric understanding of this definition is that $S$ and $\partial Q$ link if every Hilbert manifold modeled on $Q$ and sharing the same boundary intersects $S$ (see \cite{BR}). 
\end{Remark}

\qquad An useful example of linking sets, is provided in \cite{BR} and stated below. 

\begin{lemma}\label{LE}
	(See \cite{BR} Lemma 1.3) \ Let \ $e\in \partial B_1\cap E_1$ \ and $r_1>\rho>0$. If $S= \partial B_{\rho}\cap E_1$ \ and $Q= \{re:r\in[0,r_1]\}\oplus(B_{r_2}\cap E_2)$, then $S$ and $Q$ link.	
\end{lemma}

\qquad First of all, some definitions and notations introduced in \cite{BR} are required.\\

	Let $B:E\to \mathbb{R}$ be a functional. $B$ is said to be uniformly differentiable in bounded subsets of $E$, if for any $R, \varepsilon>0$, there exists  $\delta=\delta(R,\varepsilon)>0$, independent of $u$, such that
	\[
	|B(u+v)-B(u) - B'(u)v|\leq\varepsilon||v||,
	\]
	for all $u, \ u+v \in B_R$ and $||v||\leq \delta$.

\begin{lemma}\label{B'}
	Let $B:E\to \mathbb{R}$ be a functional which is weakly continuous and uniformly differentiable in bounded subsets of E. Then $B':E\to E'$ is completely continuous.
\end{lemma}
The proof follows from elementary arguments.\\

	Let $\Gamma$ denote the set of mappings $h \in C([0,1]\times E, E)$ satisfying:
	
	\hspace{-0.5cm}$(\Gamma_1)$ \ $h_t(u)= U_t(u)+K_t(u)$, where $U, K \in C([0,1]\times E, E), U_t$ is a homeomorphism of $E$ onto $E$ and $K_t$ is compact for each $t \in [0,1]$;\\
	\hspace{-0.5cm}$(\Gamma_2)$ \ $U_0(u)=u, K_0(u)=0$;\\
	\hspace{-0.5cm}$(\Gamma_3)$ \ $P_iU_t(u)= U_t(P_i(u)), \ i = 1,2;$\\
	\hspace{-0.5cm}$(\Gamma_4)$ \ $h_t$ maps bounded sets to bounded sets.\\
	\hspace{-0.5cm}In addition, for $h \in \Gamma$, let $h^j_t(u)$ denote the $j$-fold composite of $h$ with itself, i. e. $h^1_t(u)=h_t(u),$   ${h^2_t(u)=h_t(h_t(u))}$, and $h_t^j(u)=h_t(h_t^{j-1}(u))$, for $j>1$.

\qquad Now, it is convenient to state the Linking Theorem for Cerami sequences, of the Introduction.

\begin{theorem}\label{ALT}
	\textbf{(Abstract Linking Theorem)} Let $E$ be a real Hilbert space, with inner product $\big( \cdot, \cdot \big)$, $E_1$ a closed subspace of $E$ and $E_2=E_1^{\perp}$. Let $I \in C^1(E,\mathbb{R})$ satisfying:\\
	
	\hspace{-0.5cm}$(I_1) \ \ I(u)=\dfrac{1}{2}\big(Lu,u\big)+B(u),$ for all $u \in E$, where $u=u_1 +u_2 \in E_1 \oplus E_2$, $Lu=L_1u_1 + L_2u_2$ and\  $L_i: E_i \rightarrow E_i, \ i=1,2$ is a bounded linear self adjoint mapping.\\
	
	\hspace{-0.5cm}$(I_2) \ \ B$ is weakly continuous and uniformly differentiable on bounded subsets of $E$.\\
	
	\hspace{-0.5cm}$(I_3) \ $ There exist Hilbert manifolds $S,Q \subset E$, such that $Q$ is bounded and has boundary $\partial Q$, constants $\alpha > \omega$ and $v \in E_2$ such that\\
	$(i) \ S\subset v + E_1$ and $I \geq \alpha$ on $S$;\\
	$(ii) \ I\leq \omega$ on $\partial Q$;\\
	$(iii) \ S$ and $\partial Q$ link, that is satisfy de linking definition in \cite{BR}. \\
	
	\hspace{-0.5cm}$(I_4) \  \ $Setting $c$ by
	\begin{equation}\label{e9}
	c:= \displaystyle\inf_{h\in \Lambda}\sup_{u\in \bar{Q}}I(h_1(u)),
	\end{equation}
	where $\overline{Q}$ is the closure of $Q$,
	\[
	\Lambda = \left\{ 
	h \in C([0,1]\times E,E) : \begin{array}{ll} h= h^{(1)}\circ\cdot\cdot\cdot\circ h^{(m)},\ \ h^{(1)},..., \ h^{(m)} \in \Gamma, \ m\in\mathbb{N}\\
	h_t(\partial Q) \subset I^{\frac{\alpha - \omega}{2} - \beta}, \ \ \beta \in \left(0, \frac{\alpha-\omega}{2}\right)
	\end{array}
	\right\},
	\]
	and \ $I^{\lambda} \ = \ \{u \ \in E : I(u)\ \leq \ \lambda\}$, \ for \ all \ $\lambda \ \in \ \mathbb{R}$. \ If for a sequence $(u_n)$, there \ exists \ a \ constant \ $b\ >\ 0$ \ such \ that\ ${(u_n) \subset I^{-1}([c-b,c+b])}$ and $\left(1+||u_n||\right)||I'(u_n)|| \to 0$ as $n\to+\infty$, then $(u_n)$ is bounded.
	\vspace{0.5cm}\\
	\hspace{-0.5cm}Then $c\geq \alpha$, and $c$ is a critical value of $I$.\\
\end{theorem}

	Let $I \in C^1(E,\mathbb{R})$ be a functional. A sequence $(u_n)\subset E$ is said to be a Cerami Sequence or a $(C)$ sequence for short, if it satisfies:
	\[
	 sup_n|I(u_n)|< \infty, \text{ \ and \ } ||I'(u_n)||(1+||u_n||)\to 0,
	 \]
	as $n \to +\infty.$
	Given $c\in \mathbb{R}$, $(u_n) \subset E$ is said to be a Cerami sequence on level $c$ or a $(C)_c$ sequence for short, if it satisfies:
	\[
	 I(u_n)\to c, \text{ \ and \ } ||I'(u_n)||(1+||u_n||)\to 0,
	 \]
	as $n \to +\infty.$

\section{A Quantitative Deformation Lemma and Proof of the Main Result}

\qquad Inspired by \cite{BR} it is suitable to state a variant of standard Quantitative Deformation Lemma for Cerami sequences, without Cerami condition, which is necessary to prove Theorem \ref{ALT}.

\begin{lemma}\label{DL}
	\textbf{(Deformation Lemma):} \ Let $I \in C^1(E,\mathbb{R})$ satisfying $(I_1)-(I_2)$ as in Theorem \ref{ALT}. \ Then for any $R \in \mathbb{N}, \varrho >0$ and $\varepsilon \in \left(0,\frac{1}{10}\right)$, if $s:= (R+2)^2$, there exist $k \in \mathbb{N}$ and $\eta \in \Gamma$, such that:\\
	$i) \ \ I(\eta^{ks}_t(u)) \leq I(u) + \varrho$, for all $u \in \mathcal{B}_{R+2}$ and $t \in [0,1];$\\
	$ii)  \ $If $c \in \mathbb{R}$ and $||I'(w)||(1+||w||) \geq \sqrt{2\varepsilon},$ for all $w \in \mathcal{B}_{R+1}\cap I^{-1}([c-\varepsilon, c+ \varepsilon])$, then $I(\eta^{ks}(u)) \leq c - \dfrac{\varepsilon}{2},$ whenever $u \in \mathcal{B}_{\frac{R}{2}}\cap I^{-1}([c-\varepsilon,c+\varepsilon]).$
\end{lemma}
\begin{Remark}
	The mapping $\eta$ is usually determined by solving an appropriate differential equation involving $I'(\eta)$. Such an approach seems to fail here  since it does not give an $\eta$ satisfying $(\Gamma_1)-(\Gamma_3)$ which are crucial for the purpose of proving Theorem \ref{ALT}. Hence, in order to prove Lemma \ref{DL} it is necessary to argue similarly to Theorem 1.5 in \cite{BR}.
\end{Remark}
\begin{proof}
	First, choose $\chi \in C^{\infty}(\mathbb{R},\mathbb{R})$ such that $\chi(t)=1$ for $s\leq R+1$, $\chi(t)=0$ for $s\geq R+2$, $\chi'(t)<0$ for $t \in (R+1,R+2)$, and also assume $\chi(t)\leq(R+2-t)^2$ for $t \in[R+\frac{3}{2},R+2]$. For $u=u_1+u_2 \in E_1 \oplus E_2$, set $V_i(u):=\chi(||u_i||)P_iI'(u), \ i = 1,2$ and $V(u):=V_1(u)+V_2(u)$. Note that $V(u)\equiv I'(u)$ in $\mathcal{B}_{R+1}$. Moreover, by $(I_1)-(I_2)$, there is a constant $M=M(R)$ such that $||I'(u)||\leq M$, for $u \in \mathcal{B}_{R+2}$.
	With this, set \begin{equation}\label{e5}
	\bar{\varepsilon} := \dfrac{1}{Ms}\min\left(\varrho,\dfrac{\varepsilon}{2} \right).
	\end{equation}
	Since $I$ is uniformly differentiable on bounded sets, there is a $\delta= \delta(\bar{\varepsilon},R)>0$ such that
	\begin{equation}\label{e2}
	|I(u+v)-I(u)- I'(u)v|\leq \bar{\varepsilon}||v||,
	\end{equation}
	for $u, u+v \in \mathcal{B}_{R+2}$ and $||v||\leq \delta$. Assume $\delta\leq 1$ and choose $k \in \mathbb{N}$ such that
	
	\begin{equation}\label{e1}
	\dfrac{1}{k} < \min\left(\dfrac{\delta}{2M},\dfrac{1}{8\left(R+2\right)\left(1 + \displaystyle\max_{\mathbb{R}}|\chi'(s)|(||L_1||+||L_2||\right)}\right).
	\end{equation}
	Now define $\eta_t(u):= u - \dfrac{t}{k}V(u)$.\\ 
	
	\hspace{-0.4cm}\textbf{Claim.} \textit{$\eta \in \Gamma$.}
	
	\qquad Assuming this Claim and postponing its proof, it is necessary to check that $\mathcal{B}_{R+2}$ is an invariant set for $\eta_t$, for the purpose of proving $(i)-(ii)$. In fact, for $u=u_1+u_2 \in \mathcal{B}_{R+2}$, by the definitions of $\eta$ and $V$ it follows that
	\begin{equation}\label{e110}
	||P_i\eta_t(u)-u_i|| = ||(u_i-\dfrac{t}{k}V_i(u))-u_i|| = \dfrac{t}{k}||V_i(u)||\leq \dfrac{M}{k}\chi(||u_i||).
	\end{equation}
	Note that, for $u_i \in B_{R+\frac{3}{2}}$, the choice of $\chi$ implies that 
	\[
	\frac{M}{k}\chi(||u_i||)\leq \dfrac{1}{2}\leq R+2 - ||u_i||,
	\]
	via (\ref{e1}), while for $||u_i||\geq R+ \dfrac{3}{2}$, it follows that
	\[
	 \dfrac{1}{2} \geq R+2 - ||u_i|| \geq (R+2-||u_i||)^2\geq \dfrac{M}{k}\chi(||u_i||),
	 \]
	by the choice of $\chi$ and $k$. Hence, the right hand side of (\ref{e110}) does not exceed $R+2 - ||u_i||, \ i = 1,2$, which implies $\mathcal{B}_{R+2}$ is invariant for $\eta_t$. Indeed, from (\ref{e110}) and the triangular inequality, it yields
	\begin{equation}\label{ea1}
	||P_i\eta_t(u)|| \leq \dfrac{M}{k}\chi(||u_i||) + ||u_i|| \leq R+2.
	\end{equation}
	Since
	$dist(\eta_t(u),\partial \mathcal{B}_{R+2})$, the distance from $\eta_t(u)$ to $\partial \mathcal{B}_{R+2}$ satisfies 
	\[
	dist(\eta_t(u),\partial \mathcal{B}_{R+2})= \min_{i=1,2}(R+2-||P_i(\eta_t(u))||),
	\]
	thus (\ref{ea1}) implies that $\eta_t(u)\in \mathcal{B}_{R+2}$.
	
	\qquad In order to prove $i)$, observe that by (\ref{e1}) and the definitions of $\eta, \ M$ and $k$, it follows that
	\begin{equation}\label{e22}
	||\eta_t(u)-u||= ||-\dfrac{t}{k}V(u)||=\dfrac{t}{k}||V(u)|| \leq \dfrac{M}{k} <\delta,
	\end{equation}
	for all $u \in \mathcal{B}_{R+2}.$ Hence, fixing $u \in \mathcal{B}_{R+2}$ and using (\ref{e2}) with $v= -\dfrac{t}{k}V(u)$, it yields
	\begin{equation}\label{e3}
	I(u+v)= I(\eta_t(u)) \leq I(u) - \dfrac{t}{k}I'(u)V(u) + \dfrac{\bar{\varepsilon}t}{k}||V(u)||.
	\end{equation}
	Since $E_2=E^{\perp}_1$, then $P_1(I'(u)) \perp P_2(I'(u)),$ and by the definition of $V$, it follows that
	\begin{equation}\label{e4}
	I'(u)V(u) = \chi(||u_1||)||P_1(I'(u))||^2 + \chi(||u_2||)||P_2(I'(u))||^2 \geq 0.
	\end{equation}
	Thus, by (\ref{e5}), (\ref{e22}), (\ref{e3}), (\ref{e4}) and the definition of $\bar{\varepsilon}$, it follows that
	\begin{equation}\label{e6}
	I(\eta_t(u))\leq I(u) + \dfrac{\bar{\varepsilon}M}{k} \leq I(u) + \dfrac{\varrho}{Ms}\dfrac{M}{k} = I(u) + \dfrac{\varrho}{ks}.
	\end{equation}
	Provided that $\mathcal{B}_{R+2}$ is invariant under $\eta_t$, $i)$ holds by iterating (\ref{e6}) $ks$ times. In fact, iterating twice means using $\eta_t(u)$ instead of $u$ in (\ref{e6}), and after that using again (\ref{e6}), but for $u$, it yields
	\[
	I(\eta_t^2(u)) = I(\eta_t(\eta_t(u))) \leq I(\eta_t(u)) + \dfrac{\varrho}{ks} \leq I(u) + \dfrac{2\varrho}{ks}.
	\]
	Then, after $ks$ iterations, it yields
	\[
	I(\eta^{ks}_t(u)) \leq I(u) + \dfrac{ks\varrho}{ks}\leq I(u) + \varrho,
	\]
	and $i)$ is proved.
	
	\qquad In order to prove $ii)$, take $u\in \mathcal{B}_{\frac{R}{2}}\cap I^{-1}([c- \varepsilon, c + \varepsilon])$. Three cases are considered:\\
	
	\hspace{-0.5cm}$\textbf{Case I:} \ \eta_1^j(u) \in \mathcal{B}_{R+1}\cap I^{-1}([c-\varepsilon, c + \varepsilon])$ for $1\leq j \leq ks$. By definition, $V(u)= I'(u)$ in $\mathcal{B}_{R+1}$, then fixing $j, \ 1\leq j \leq ks$,  the definition of $\eta_1$ yields 
	\[
	\eta_1^j(u)-\eta_1^{j-1}(u)= -\dfrac{1}{k}V(\eta_1^{j-1}(u))=-\dfrac{1}{k}I'(\eta_1^{j-1}(u)).
	\]
	Then, using (\ref{e3}) for $\eta^{j}_1(u)$ and $\eta_1^{j-1}(u)$ instead of $u+v$ and $u$, by the definition of $M$, and due to (\ref{e5}), it yields
	\begin{equation}\label{jj-1}
	I(\eta_1^j(u))-I(\eta_1^{j-1}(u))  \leq -\dfrac{1}{k}||I'(\eta_1^{j-1}(u))||^2 + \dfrac{1}{k}\bar{\varepsilon}M \leq -\dfrac{1}{k}||I'(\eta_1^{j-1}(u))||^2 + \dfrac{\varepsilon}{2ks}.
	\end{equation}
	Then, by the telescopic sum, and using (\ref{jj-1}) for all $1\leq j \leq ks$, it follows that
	\begin{equation}\label{ts}
	I(\eta_1^{ks}(u)) - I(u)= \displaystyle\sum_{j=1}^{ks}\left[I(\eta_1^j(u))-I(\eta_1^{j-1}(u))\right] \leq \sum_{j=1}^{ks}\left[-\dfrac{1}{k}||I'(\eta_1^{j-1}(u))||^2+\dfrac{\varepsilon}{2ks}\right].
	\end{equation}
	Provided that $(R+2)||I'(u)||\geq (1+||u||)||I'(u)||\geq \sqrt{2\varepsilon}$ from assumption, setting  $\varepsilon_s := \dfrac{\varepsilon}{s}$, it follows that
	$||I'(u)||^2\geq 2 \varepsilon_s,$
	for all $u \in \mathcal{B}_{R+1}$. Thus, (\ref{ts}) yields
	\begin{equation}\label{e7}
	I(\eta_1^{ks}(u)) - I(u) \leq \displaystyle\sum_{j=1}^{ks}\left[-\dfrac{2\varepsilon_s}{k}+\dfrac{\varepsilon_s}{2k}\right] = - \dfrac{3\varepsilon}{2}.
	\end{equation}
	Therefore, since $I(u)\leq c + \varepsilon$, (\ref{e7}) implies that $I(\eta_1^{ks}(u)) \leq I(u) - \dfrac{3\varepsilon}{2} \leq c - \dfrac{\varepsilon}{2}$,  and the result holds for first case.\\
	
	\hspace{-0.5cm}$\textbf{Case II:} \ \eta_1^j(u) \in \mathcal{B}_{R+1}\cap I^{-1}([c-\varepsilon, c + \varepsilon])$ for $1\leq j \leq m-1 $, but $\eta_1^m(u) \notin I^{-1}([c-\varepsilon,c+\varepsilon]), $ for some $1\leq m \leq ks$. From (\ref{jj-1}) and since $||I'(u)||^2\geq 2 \varepsilon_s,$
	for all $u \in \mathcal{B}_{R+1}$, it follows that 
	\[I(\eta_1^j(u))- I(\eta_1^{j-1}(u)) \leq - \dfrac{2\varepsilon_s}{k}+\dfrac{\varepsilon_s}{2k} = -\frac{3\varepsilon_s}{2k},
	\]
	hence, $I(\eta_1^j(u)) < I(\eta_1^{j-1}(u))$ for $1\leq j \leq m$. Then, $\eta_1^m(u)\notin I^{-1}([c-\varepsilon, c+\varepsilon])$ and $I(\eta_1^m(u)) < I(\eta_1^{m-1}(u))$ implies that $I(\eta_1^m(u)) < c-\varepsilon$. Thus, using again a telescopic sum, it follows that
	\begin{eqnarray}\label{c2ts}
	I(\eta_1^{ks}(u))&=& I(\eta_1^m(u)) + \displaystyle\sum_{j=m +1}^{ks}\left[I(\eta_1^j(u))- I(\eta_1^{j-1}(u))\right] \nonumber\\
	&\leq& c-\varepsilon + \displaystyle\sum_{j=m +1}^{ks}\left[I(\eta_1^j(u))- I(\eta_1^{j-1}(u))\right].
	\end{eqnarray}
	Fixing $j, \ m+1\leq j\leq ks$, and using (\ref{e3}) as in (\ref{jj-1}), it yields
	\begin{equation}\label{jj-12}
	I(\eta_1^j(u))-I(\eta_1^{j-1}(u)) \leq  -\dfrac{1}{k}||I'(\eta_1^{j-1}(u))||^2 + \dfrac{\varepsilon}{2ks} \leq \dfrac{\varepsilon}{2ks}.
	\end{equation}
	Replacing (\ref{jj-12}) in (\ref{c2ts}) for all $m+1\leq j\leq ks$, it follows that
	\begin{eqnarray*}
		I(\eta_1^{ks}(u))&\leq& 
		c-\varepsilon + \displaystyle\sum_{j=m +1}^{ks}\left[\dfrac{\varepsilon}{2ks}\right]\\
		&=& c-\varepsilon + \left(\dfrac{ks - m}{ks}\right)\dfrac{\varepsilon}{2}\\
		&\leq& c -\dfrac{\varepsilon}{2}.
	\end{eqnarray*}
	Therefore, the result also holds in this case.\\
	
	\hspace{-0.5cm}$\textbf{Case III:} \ \eta_1^j(u) \in \mathcal{B}_{R+1}\cap I^{-1}([c-\varepsilon, c + \varepsilon])$ for $1\leq j \leq m-1 $, but $\eta_1^m(u)  \notin \mathcal{B}_{R+1}$ for some $1\leq m \leq ks$. Since $u \in \mathcal{B}_{\frac{R}{2}}$, it follows that
	\[
	||u|| + \dfrac{R}{2} + 1 \leq R + 1 \leq ||\eta_1^m(u)||
	\]
	and hence, by the triangular inequality and the telescopic sum, it follows that
	\begin{eqnarray}\label{e8}
	\dfrac{R+2}{2} &\leq& ||\eta_1^m(u)|| - ||u|| \nonumber\\
	&\leq& ||\eta_1^m(u) - u ||\nonumber\\
	&\leq& \displaystyle\sum_{j=1}^{m}||\eta_1^j(u)- \eta_1^{j-1}(u)||.
	\end{eqnarray}
	By the definition of $\eta_1$, and that $V(u)=I'(u)$ in $\mathcal{B}_{R+1}$, it follows that 
	\begin{equation}\label{3e8}
	||\eta_1^j(u)- \eta_1^{j-1}(u)||= \dfrac{1}{k}||V(\eta_1^{j-1}(u))||= \dfrac{1}{k}||I'(\eta_1^{j-1}(u))||,
	\end{equation}
	for all $1\leq j\leq m$, since $\eta_1^{j}(u)\in \mathcal{B}_{R+1}$ for all $1\leq j\leq m-1$. Hence, replacing (\ref{3e8}) in (\ref{e8}) and applying Holder's Inequality for finite sums, it yields
	\begin{eqnarray}\label{2e8}
	\dfrac{R+2}{2} &\leq& \dfrac{1}{k}\displaystyle\sum_{j=1}^{m}||I'(\eta_1^{j-1}(u))||\nonumber\\
	&\leq& \dfrac{m^{\frac{1}{2}}}{k}\left[\displaystyle\sum_{j=1}^{m}||I'(\eta_1^{j-1}(u))||^2\right]^{\frac{1}{2}}.
	\end{eqnarray}
	From (\ref{jj-1}), it follows that $\dfrac{1}{k}||I'(\eta_1^{j-1}(u))||^2 \leq -[I(\eta_1^j(u))-I(\eta_1^{j-1}(u))] + \dfrac{\varepsilon}{2ks}$, which yields
	\begin{equation}\label{c3}
	||I'(\eta_1^{j-1}(u))||^2 \leq k\left(I(\eta_1^{j-1}(u))-I(\eta_1^{j}(u))\right) + \dfrac{\varepsilon_s}{2},
	\end{equation}
	for all $1\leq j \leq m$. Thus, using (\ref{c3}) in (\ref{2e8}), it yields
	\[
	\dfrac{R+2}{2}\leq \dfrac{m^{\frac{1}{2}}}{k}\left[\displaystyle\sum_{j=1}^{m}k\left(I(\eta_1^{j-1}(u))-I(\eta_1^{j}(u))\right) + \dfrac{\varepsilon_s}{2}\right]^{\frac{1}{2}}= \dfrac{m^{\frac{1}{2}}}{k}\left[k\Big(I(u)-I(\eta_1^{m}(u))\Big) + \dfrac{m\varepsilon_s}{2}\right]^{\frac{1}{2}}.
	\]
	Squaring both sides, it follows that
	\begin{eqnarray}\label{c3e9}
	\dfrac{s}{4} &=& \left(\dfrac{R+2}{2}\right)^2 \nonumber\\
	&\leq&\dfrac{m}{k^2}\left[k\Big(I(u) - I(\eta_1^m(u))\Big) + \dfrac{m\varepsilon_s}{2}\right]\nonumber\\
	&=& \dfrac{m}{k}\left[I(u) - I(\eta_1^m(u)) + \dfrac{m\varepsilon}{2ks}\right]\nonumber\\
	&\leq& \dfrac{m}{k}\left[c + \varepsilon - I(\eta_1^{m}(u)) + \dfrac{\varepsilon}{2}\right].
	\end{eqnarray}
	Multiplying both sides of (\ref{c3e9}) by $\dfrac{k}{m}$, it yields
	\[
	\dfrac{1}{4} \leq \dfrac{ks}{4m} \leq c + \dfrac{3\varepsilon}{2} - I(\eta_1^{m}(u)),
	\]
	which implies $ I(\eta_1^m(u)) \leq c + \dfrac{3\varepsilon}{2} - \dfrac{10\varepsilon}{4} = c - \varepsilon$, since $0 < \varepsilon < \dfrac{1}{10}$ by assumption. Now, as in Case $II$, using a telescopic sum and (\ref{jj-12}), it yields
	\begin{eqnarray*}
		I(\eta_1^{ks}(u))&=& I(\eta_1^m(u)) + \displaystyle\sum_{j=m +1}^{ks}\left[I(\eta_1^j(u))- I(\eta_1^{j-1}(u))\right]\\
		&\leq& c-\varepsilon + \displaystyle\sum_{j=m +1}^{ks}\left[\dfrac{\varepsilon}{2ks}\right]\\
		&=& c-\varepsilon + \left(\dfrac{ks - m}{ks}\right)\dfrac{\varepsilon}{2}\\
		&\leq& c -\dfrac{\varepsilon}{2}.
	\end{eqnarray*}
	Therefore $ii)$ also holds for the third case.
	
	\qquad Finally, to finish the proof of this lemma, it is left to prove the Claim. Denoting by $Id:E\to E$ the identity map in $E$, since $P_iV(u)= V_i(u)=\chi(||u_i||)P_iI'(u)= \chi(||u_i||)(L_iu_i +P_iB'(u))$, it holds that
	\[
	P_i\eta_t(u)= P_i(u-\dfrac{t}{k}V(u)) = \left(Id- \dfrac{t}{k}\chi(||u_i||)L_i\right)u_i - \dfrac{t}{k}\chi(||u_i||)P_iB'(u).
	\]
	Hence, it is suitable to set
	\[
	U_t(u):= \sum_{i=1,2}\left(Id-\dfrac{t}{k}\chi(||u_i||)L_i\right)u_i,
	\]
	and
	\[
	K_t(u):= -\dfrac{t}{k}\sum_{i=1,2}\chi(||u_i||)P_iB'(u).
	\]
	In fact, observe that given $u_n \rightharpoonup u$ in $E$, $(I_2)$  and Lemma \ref{B'} imply that $B'(u_n)\to B'(u)$ in $E'$. Since $P_i$ is continuous, it follows that $K_t(u_n) \to K_t(u)$, thus $K_t$ is completely continuous, and therefore is compact for all $t\in[0,1]$. Moreover, by the definitions of $U_t$ and $K_t$, it is clear that $\eta$ satisfies $(\Gamma_2)-(\Gamma_3)$. Via $(I_1)-(I_2)$ there is a constant $C=C(r)$ such that $||I'(u)||\leq C$ for all $u\in B_r$, hence by the definition of $\eta_t$, it yields
	\begin{eqnarray*}
		||\eta_t(u)|| \leq r +||I'(u)||	\leq r + C,
	\end{eqnarray*}
	and so $\eta_t$ maps bounded sets to bounded sets, namely $\eta$ satisfies $(\Gamma_4)$. Lastly, it suffices to show that $U_t$ is a homeomorphism of $E$ onto $E$, in order to complete the verification of $(\Gamma_1)$. By $(\Gamma_3)$, it suffices to show that $P_iU_t$ is a homeomorphism of $E_i$ onto $E_i, \ i=1,2$. Let $u,v \in E_i$,  for any $t\in[0,1]$, by the definition of $\chi$, if $||u||,||v||\geq R+2$, then
	\[
	\dfrac{t}{k}\Big|\Big|\chi(||u||)L_iu-\chi(||v||)L_iv\Big|\Big|= 0 \leq \dfrac{1}{2}||u-v||.
	\]
	Hence, without loss of generality, suppose that $||v||\leq R+2$. By the definitions of $\chi$ and $k$, and via the Mean Value Theorem,  for any $t\in[0,1]$, it yields
	\begin{eqnarray}\label{C1}
	\dfrac{t}{k}\Big|\Big|\chi(||u||)L_iu-\chi(||v||)L_iv\Big|\Big|
	&\leq& \dfrac{1}{k}||L_i||\ ||u-v|| + \dfrac{1}{k}||L_i|| \ |\chi(||u||)-\chi(||v||)| \ ||v|| \nonumber\\
	&\leq& \dfrac{1}{k}||L_i||(R+2)\left[1 +  \max_{\mathbb{R}}|\chi'(s)|\right] ||u-v|| \nonumber\\
	&\leq& \dfrac{1}{2}||u-v||.
	\end{eqnarray} 
	Thus, (\ref{C1}) holds for all $u,v \in E_i$. Now, for each \ $w \in E_i$ fixed, \ set \ $\mathcal{L}_w(u):E_i \to E_i$, given by
	$\mathcal{L}_w(u):=\dfrac{t}{k}\chi(||u||)L_iu + w$. Note that  due to  (\ref{C1}), $\mathcal{L}_w(u)$ is a contraction on $E_i$. Then, it follows from the contracting mapping theorem that $\mathcal{L}_w(u)$ has a unique fixed point $u_w$. Therefore, $u_w\in E_i$ is the unique such that $\mathcal{L}_w(u_w)= u_w$, which implies that $P_iU_t(u_w)= w$ is an one-to-one correspondence, and hence $P_iU_t$ is bijection. Furthermore, (\ref{C1}) yields
	\begin{eqnarray}\label{C2}
	||P_iU_t(u)-P_iU_t(v)|| &=& ||(u-v) - \dfrac{t}{k}(\chi(||u||)L_iu - \chi(||v||)L_iv)|| \nonumber\\
	&\geq&  ||u-v|| - \dfrac{1}{2}||u-v||= \dfrac{1}{2}||u-v||,
	\end{eqnarray}
	which implies that $(P_iU_t)^{-1}$ is continuous. Since $P_iU_t$ is continuous by definition, it ensures that $P_iU_t$ is a homeomorphism of $E_i$ onto $E_i$. 
	Consequently, $\eta$ satisfies $(\Gamma_1)$ and finally the claim is proved.
\end{proof}
\qquad The following lemma gives a significant information about the level $c$.
\begin{lemma}\label{prop}(See \cite{BR} Proposition 1.17.)
	If $I$ satisfies $(I_3)$, then $c\geq \alpha.$
\end{lemma}
\begin{proof}
	For the sake of completeness, this proof is included here. In fact, it suffices to show that
	\begin{equation}\label{s1l3e1}
	h_1(\bar{Q})\cap S \not= \emptyset,
	\end{equation}	
	for all $h \in \Lambda.$ In fact, if (\ref{s1l3e1}) holds, there is a $y\in h_1(\bar{Q})\cap S$, hence
	\begin{equation}\label{s1l3e2}
	\sup_{u\in\bar{Q}}I(h_1(u))\geq I(y)\geq \inf_{w\in S}I(w)\geq \alpha,
	\end{equation}
	due to $(I_3) \ (i)$. Since (\ref{s1l3e2}) holds for all $h\in \Lambda$, the definition of $c$ in $(I_4)$ yields $c\geq \alpha$. 
	The proof of (\ref{s1l3e1}) follows from the stronger assertion that
	\begin{equation}\label{s1l3e3}
	h_t(\bar{Q})\cap S \not= \emptyset,
	\end{equation}
	for all $h \in \Lambda$ and $t \in [0,1]$. Since $S-v \subset E_1$ via $(I_3) \ (i)$, then (\ref{s1l3e3}) is equivalent to finding, for each $t \in [0,1]$, a $u \in \bar{Q}$ such that
	\begin{eqnarray}\label{s1l3e4}
	&&P_1h_t(u) \in S - v \nonumber\\
	&&P_2h_t(u) = v.
	\end{eqnarray}
	In order to solve (\ref{s1l3e4}), it \ is necessary \ to convert it into an equivalent \ problem \ to which \ the linking geometry \ hypotheses \ can be applied. \ Suppose first that $h\in \Lambda$ with \ the corresponding $m=1$. \ Letting \ $u = u_1+u_2 \in E_1 \oplus E_2$ as usual, by $(\Gamma_1)$ and $(\Gamma_3)$, (\ref{s1l3e4}) becomes
	\begin{eqnarray}\label{s1l3e5}
	&& (i) \ P_1h_t(u) \in S-v \nonumber\\
	&& (ii) \ P_2h_t(u) = U_t(u_2)+ P_2K_t(u) = v.
	\end{eqnarray}
	More generally, suppose (\ref{s1l3e5}) $(ii)$ replaced by
	\begin{equation}\label{s1l3e6}
	P_2h_t(u)= P_2Z_t(u),
	\end{equation}
	where $Z_t(u)$ is an arbitrary compact operator with $Z_0(u)=v$. Note that  in (\ref{s1l3e5}), $Z_t(u)=v$ is the compact constant operator, for all $t\in [0,1]$. Again via $(\Gamma_1)$ and $(\Gamma_3)$ (\ref{s1l3e6}) is equivalent to
	\[
	U_t(u_2) = -P_2\Big(K_t(u) - Z_t(u)\Big),
	\]
	which is equivalent to
	\begin{eqnarray}\label{s1l3e7}
	u_2 = U_t^{-1}\left(-P_2\Big(K_t(u) - Z_t(u)\Big)\right) \equiv P_2Y_t(u),
	\end{eqnarray}
	where $Y_t$ is compact, due to the compactness of $K_t$ and $Z_t$, and $Y_0(u)=v$, since 
	\[
	U_0^{-1}\left(-P_2\Big(K_0(u) - Z_0(u)\Big)\right)=-P_2\Big( - Z_0(u)\Big) = v.
	\]
	Now, suppose by induction that (\ref{s1l3e6}) is equivalent to
	\begin{equation}\label{s1l3e8}
	u_2= P_2Y_t(u),
	\end{equation} 
	with $Y_t$ compact and $Y_0(u)=v$, whenever $h\in \Lambda$ with the corresponding $m=n-1$. Then, let $h \in \Lambda$ with $m=n$ so $h= h^{(1)}\circ\cdot\cdot\cdot\circ h^{(m)}$ and let $\hat{h}= h^{(2)}\circ\cdot\cdot\cdot\circ h^{(m)}$. Hence $h= h^{(1)}\circ\hat{h}$ and again by $(\Gamma_1)$ and $(\Gamma_3)$, the equation 
	\[
	v=P_2h_t(u)= P_2\Big(U_t^{(1)}+K_t^{(1)}\Big)\hat{h}_t(u)= U_t^{(1)}P_2\hat{h}_t(u) + P_2K_t^{(1)}\hat{h}_t(u)
	\]
	is equivalent to
	\begin{equation}\label{s1l3e9}
	P_2\hat{h}_t(u) = (U_t^{(1)})^{-1}(-P_2K_t^{(1)}(\hat{h}_t(u))+v) =:P_2\hat{Z}_t(u),
	\end{equation}
	where $h^{(1)}= U^{(1)}+K^{(1)}$ and since  $K_t^{(1)}$ is  compact, $\hat{Z}_t$ given by the right hand side of (\ref{s1l3e9}) is compact and $P_2\hat{Z}_0(u)= v$, since $K_0=0$. Thus, by induction hypothesis there is a compact ${Y}_t$ such that (\ref{s1l3e9}) is equivalent to solving (\ref{s1l3e8}).
	
	\qquad Now set $\Phi_t(u)= P_1h_t(u)+ u_2-P_2Y_t(u)+v$, and note that $\Phi \in \Sigma$, since 
	\[
	P_2\Phi_t= u_2 - (P_2Y_t -v),
	\] and
	\[
	\Phi_0(u)= P_1h_0(u)+ u_2-P_2Y_0(u)+v=P_1u+u_2 -v+v= u.
	\]
	In addition, $P_1\Phi_t= P_1h_t$ and provided that $P_2h_t=v$ is equivalent to (\ref{s1l3e8}),  due to all remarks above, it follows that $P_2\Phi_t=v$ is equivalent to $P_2h_t=v$, by the definition of $\Phi_t$. Therefore, $\Phi_t(u)\in S$ if and only if $h_t(u)\in S$, hence to obtain (\ref{s1l3e3}) and complete the proof, it is only necessary to show that
	\begin{equation}\label{s1l3e10}
	\Phi_t(Q)\cap S\not= \emptyset,
	\end{equation}
	for all $t \in[0,1]$. Since $\Phi \in \Sigma$ and via $(I_3) \ (iii)$, $S$ and $\partial Q$ link, then (\ref{s1l3e10}) holds if
	\begin{equation}\label{s1l3e11}
	\Phi_t(\partial Q)\cap S = \emptyset.
	\end{equation} 
	Suppose the contrary, so there is a $u\in \partial Q$ and $t \in [0,1]$ such that $\Phi_t(u) \in S$. Then $h_t(u) \in S$, but  $h_t(\partial Q)\subset I^{\frac{\alpha+\omega}{2}-\beta}$,
	since $h \in \Lambda$. On the other hand, $S\cap I^{\frac{\alpha+\omega}{2}-\beta} = \emptyset$, due to $(I_3) \ (i)$ and provided that $\beta \in \left(0, \dfrac{\alpha-\omega}{2}\right)$, hence it yields a contradiction. Thus (\ref{s1l3e11}) is satisfied and the proof of the lemma is complete. 
\end{proof}	

\qquad Now,under the knowledge of  all results in previous sections, Theorem \ref{ALT} can be finally proved.\\ 

\begin{proof}[Proof of Theorem \ref{ALT}.] First, since the identity map $h(u)=u$ is in $\Lambda$, then $c < +\infty$ in view of $(I_2)$ and $(I_4)$. Moreover, $c\geq\alpha$ by Lemma \ref{prop}. Now suppose that $c$ is not a critical value of $I,$ then $I'(u)\not= 0$, for all $u \in I^{-1}(c),$ hence there exists $\varepsilon>0$ such that
\begin{equation}\label{e10}
\left(1+||u||\right)||I'(u)|| \geq \sqrt{2\varepsilon}, \text{ \ for all \ } u \in I^{-1}([c-\varepsilon,c+\varepsilon]).
\end{equation}
If not, for all $n \in \mathbb{N}$ there exists a sequence of positive $\varepsilon_n \to 0$ and $u_n \in I^{-1}([c-\varepsilon_n,c+\varepsilon_n])$ such that 
\[
\left(1+||u_n||\right)||I'(u_n)|| < \sqrt{2\varepsilon_n}.
\]
From $(I_4)$ this sequence is bounded and then it possesses a weakly convergent subsequence, still denoted by $(u_n)$, namely $u_n \rightharpoonup u$ as $n \to +\infty$, for some $u \in E$. By $(I_2)$ and Lemma \ref{B'} one has $\ B'(u_n) \to B'(u)$ along this subsequence and by assumption, $I'(u_n) \to 0$, then \ it \ follows  that  $Lu_n = I'(u_n)-B'(u_n) \to -B'(u)$ as $n \to +\infty$. \ On the other hand, \ $Lu_n$ \ also \ converges \ weakly \ to \ $Lu$ \ along \ this \ subsequence. Hence $Lu= -B'(u)$ and \ \ $Lu_n \to Lu$ \ \ strongly, \ then \ \ $I'(u)=Lu + B'(u)=0$. \ \ Since $I(u_n)\to c$, again by \ \ $(I_2)$ \ \ it \ follows \ that \ \ $I(u_n) = \dfrac{1}{2}\left(Lu_n,u_n\right) + B(u_n) \to I(u) = c$. But it means that $c$ is a critical value of $I$, contrary to assumption. Thus there exists an $\varepsilon$ as desired in (\ref{e10}). It can further be assumed $\varepsilon < \dfrac{1}{10}$.
By the definition of infimum, choose an $h \in \Lambda$ with corresponding $\beta$ such that
\begin{equation}\label{e11}
c\leq \sup_{u\in \bar{Q}}I(h_1(u)) \leq c+\varepsilon \text{\ and \ } h_t(\partial Q) \subset I^{\frac{\alpha - \omega}{2} - \beta}.
\end{equation}
Since $h \in \Lambda$, $h_t$ maps bounded sets on bounded sets, due to the definition of $\Gamma$, hence $h_1(\bar{Q})$ is bounded. Therefore, there is an $R \in \mathbb{N}$ such that $h_1(\bar{Q}) \subset B_{\frac{R}{2}}.$ By Lemma \ref{DL}, with $\varrho= \dfrac{1}{2}\min(\beta,\varepsilon)$, there exist $\eta \in \Gamma$ and $k \in \mathbb{N}$ such that $\eta_t^{ks}$ satisfies $i)$ and $ii)$ of that lemma. Let $g_t(u) = \eta_t^{ks}(h_t(u))$, provided that $h_t(\partial Q) \subset I^{\frac{\alpha - \omega}{2}- \beta}$, in view of $i)$ Lemma \ref{DL}, $g_t(\partial Q) \subset I^{\frac{\alpha - \omega}{2}- \frac{\beta}{2}}$. Hence $g \in \Lambda$ and from (\ref{e9}) it follows that
\begin{equation}\label{e12}
c\leq \sup_{u\in \bar{Q}}I(g_1(u)).
\end{equation}
From (\ref{e11}),  $I(h_1(u)) \leq c+\varepsilon$ for all $u \in \bar{Q}.$ Thus, if $h_1(u)\in I^{-1}([c-\varepsilon,c+\varepsilon])$, by (\ref{e10}) it is possible to apply $ii)$ of the Lemma \ref{DL} to conclude that $g_1(u) \in I^{c - \frac{\varepsilon}{2}}.$ On the other hand, if $h_1(u) \in I^{c-\varepsilon},$ then $i)$ of the Lemma \ref{DL} yields
\[
I(g_1(u)) = I(\eta_1^{ks}(h_1(u))) \leq I(h_1(u)) + \varrho \leq c-\varepsilon + \dfrac{\varepsilon}{2} = c - \dfrac{\varepsilon}{2},
\]
thus $g_1(u) \in I^{c-\frac{\varepsilon}{2}}$ by the choice of $\varrho$. Consequently, it follows that
\begin{equation}\label{e13}
\sup_{u\in \bar{Q}}I(g_1(u)) \leq c-\dfrac{\varepsilon}{2},
\end{equation}
which contradicts (\ref{e12}) and the theorem is proved.
\end{proof}


\section{Application to Asymptotically Linear Schr\"odinger Equations in $\mathbb{R}^N$}

\qquad This section introduce two applications for the abstract critical point theorem developed previously. The main difference between them is how to obtain the linking geometry, based on their spectra, which are very different to each other.

\qquad First, consider problem $(P)$
\begin{equation}\label{P}
-\Delta u + V(x)u = g(x,u) \text{ \ in \ } \mathbb{R}^N,
\end{equation}

\hspace{-0.5cm}for  $N\geq3$, where the potential $V$ satisfies $(V_1)-(V_2)$ as stated in the Introduction.\\

\begin{Remark} \label{s2r1}
	In view of hypothesis $(V_1)$, $V(x)$ is bounded and $\sigma_{ess}(A)= [V_{\infty}, + \infty)$ \ (see \ \cite{St} \ Theorem \ 3.15, \ page \ 44), \ \ hence \ \ $\sigma(A)\cap(-\infty,V_{\infty})= \sigma_d(A)\cap(-\infty, V_{\infty}).$ Furthermore, hypothesis $(V_2)$ implies that either $0 \notin \sigma(A)$ or $0\in \sigma_d(A)$, since $0\in (\sigma^-,\sigma^+) $ is an isolated point and the essential spectrum $[V_\infty,+\infty)$ does not have isolated points, hence $0 \notin \sigma_{ess}(A)= [V_{\infty}, +\infty)$, which implies that $V_{\infty}>0$, namely $V_\infty$ is positive, therefore assumption $V_{\infty}>0$ in $(V_1)$ is redundant. With this in hand, it is possible to introduce by means of operator $A$ an equivalent norm $||\cdot||$ to the usual norm $||\cdot||_{H^1(\mathbb{R}^N)}$, in $H^1(\mathbb{R}^N)$ (see \cite{CT}, Lemma 1.2). Thus, \ $E= \Big(H^1(\mathbb{R}^N),||\cdot||\Big)$ will be the Hilbert space used in order to apply Theorem \ref{ALT}. Finally, from $(V_2)$ $ \emptyset \not= \sigma(A)\cap(-\infty,0)= \sigma_{d}(A)\cap(-\infty, 0), $ i. e., operator $A$ has negative eigenvalues. Furthermore, this set is finite (see \cite{EK} Theorem 30, page 150).
	An example satisfying $(V_1)-(V_2)$ is given by a continuous $V(x)$ such that
	\[
	 V(x) = \left\{
	\begin{array}{lll}
	- V_0 \ \ \ \ \ \ \ \ \ |x|< R \\
	\\
	\ \ V_\infty \ \ \ \ \ \ \ \  |x|>2R,\\
	\end{array}
	\right.
	\]
	where $V_0> \dfrac{\lambda_1(1)}{R^2}>0$ is a constant and $\lambda_1(1)$ is the first eigenvalue of the operator $(-\Delta, H_0^1(B_1(0)))$.
\end{Remark}

\qquad Henceforth consider the case where $g(x,s) = h(x)f(s),$ and $h$ satisfies $(h_1)$.
Furthermore, $f$ is asymptotically linear satisfying $(f_1)-(f_3)$.\\
\begin{Remark}\label{s2r2}
	Setting $||h||_{L^\infty(\mathbb{R}^N)}:= h_{\infty}$, observe  that assumption $(h_1)$ implies that $0<h_{\infty}<+\infty$. In addition, $(f_2)$ implies that $\displaystyle\lim_{s \to +\infty}F(s)=+\infty$ and $\displaystyle\lim_{s\to +\infty}\dfrac{f(s)}{s}=a$. Moreover, due to assumptions $(f_1)-(f_2)$ there exists $\kappa>0$ such that $|f(s)|\leq \kappa|s|$ for all $s \in \mathbb{R}$ and $a\leq\kappa$. An example satisfying $(f_1)-(f_3)$ but not with $\dfrac{f(s)}{s}$ increasing, is a continuous $f(s)$ such that
	\[
	f(s)= \left\{
	\begin{array}{lll}
	\dfrac{s^7 - \frac{3}{2}s^5 + 2 s^3}{1 + s^6} \ \ \ \ \ |s| < 5, \\
	\\
	\ \ \ \ \ \ \dfrac{s^3}{1+s^2} \ \ \ \ \ \ \ \ \ |s|>10.\\
	\end{array}
	\right.
	\]	
\end{Remark}
\vspace{0.5cm}

\qquad Let $I:E \to \mathbb{R}$ be the energy functional associated with problem $(P)$ in (\ref{P}), which is given by 
\begin{equation}\label{I}
I(u)= \dfrac{1}{2}\displaystyle\int_{\mathbb{R}^N}\Big(|\nabla u|^2+V(x)u^2\Big)  - \int_{\mathbb{R}^N}h(x)F(u)dx,
\end{equation}
for all $u \in E$. As observed in Remark \ref{s2r1}, the set of eigenvalues $\sigma_d(A)\cap(-\infty,0)$ is finite, then one can denote it by $\{\lambda_i\}_{i=1}^j$, for some $j \in \mathbb{N}$, counting multiplicities, and in addition, denote by $\varphi_i\in E$ the eigenfunction associated with $\lambda_i$, for $i=1,...,j$ and then set $E^-:= span
\{\varphi_i\}_{i=1}^j$. Moreover, setting $E^0:= \ker(A)$, if $0 \notin \sigma(A)$, then $E^0=\{0\}$, if not, then $0\in \sigma_d(A)$, hence $E^0$ is finite dimensional. Thus, $E^-\oplus E^0$, it is a finite dimensional subspace of $E$ and setting  $E^+:=(E^-\oplus E^0)^{\perp}$, it is the subspace of $E$ in which operator $A$ is positive definite. With this, $E=E^+\oplus E^-\oplus E^0 $ and every function $u\in E$ can be uniquely written as $u=u^++u^-+ u^0$, with $u^+\in E^+$, $u^0 \in E^0$ and $u^- \in E^-$. Furthermore, as in \cite{CT} operator $A$ induce an equivalent norm $||\cdot||$ to the standard $H^1(\mathbb{R}^N)$-norm and a corresponding inner product $\big(\cdot,\cdot\big)$ in $E$ given by
\[
||u||^2:=  (Au^+,u^+)_2 - (Au^-,u^-)_2 + ||u^0||^2_2,
\]
and 
\begin{equation}\label{IP}
\big( u,v \big) = \left\{
\begin{array}{lllllll}
\ \ \ \displaystyle\int_{\mathbb{R}^N}\Big(\nabla u(x) \nabla v(x) + V(x)u(x)v(x)\Big)dx = (Au,v)_2 \ \ \ \ \ \ \ \ \ \text{if} \ \ u, v \in E^+, \\
\\
\ \ \ -\displaystyle\int_{\mathbb{R}^N}\Big(\nabla u(x) \nabla v(x) + V(x)u(x)v(x)\Big)dx = -(Au,v)_2\ \ \ \ \text{if} \ \ u,v \in E^-,\\
\\
\ \ \ (u,v)_2 \ \ \ \ \ \ \ \ \ \ \ \ \ \ \ \ \ \ \  \ \ \ \ \ \ \  \ \ \ \ \  \ \ \ \ \ \ \ \ \ \ \ \ \ \ \ \ \ \ \ \ \ \ \ \ \ \ \ \ \ \ \ \ \ \ \ \ \text{if} \ \ u,v \in E^0,\\
\\
\ \ \	0 \ \ \ \ \ \ \ \ \ \ \ \ \ \ \ \ \ \ \ \ \ \ \ \ \ \ \ \ \ \ \ \ \ \ \ \ \  \  \ \ \ \ \ \ \ \ \ \ \ \ \ \ \ \ \ \ \ \text{if} \ \ u \in E^j, \ v \in E^k, j\not=k,
\end{array}
\right.
\end{equation}
for $j,k \in \{+,-,0\}$. Henceforth, the Hilbert space used in this application is $E= \left(H^1(\mathbb{R}^N),||\cdot||\right)$ and in addition, it is possible to write
\[
I(u)= \dfrac{1}{2}\Big(||u^+||^2 - ||u^-||^2\Big) - \int_{\mathbb{R}^N}h(x)F(u(x))dx,
\]
for all $u=u^+ + u^-+ u^0 \in E.$ Note that the orthogonality among $E^+, E^-, E^0$, ensures that $u^+, u^- , u^0$ are also orthogonal in $L^2(\mathbb{\mathbb{R}^N}).$ As usual, $||u||^2_2= (u,u)_2$ denotes the norm and inner product in $L^2(\mathbb{\mathbb{R}^N})$, then $(u^j,u^k)_2 = 0, j\not=k$ and $j, k \in \{+,-,0\}$.

\qquad Denoting by $\{\mathcal{E}(\lambda)\}$ the spectral family of operator $A$, in view of the spectral theory (see \cite{BS}, Supplement S1.1; see also \cite{P} Chapter 3) it is possible to define $E_2:= \mathcal{E}(0)E=E^-\oplus E^0$ and $E_1 := (I-\mathcal{E}(0))E$. Furthermore, $\mathcal{E}(0) = \mathcal{E}(\lambda)$, for all $0<\lambda<\sigma^+$, by the definition of $\sigma^+$ in $(V_2)$, then $E_1 = (I-\mathcal{E}(\lambda))E$, for all $0<\lambda<\sigma^+$. Hence, by \cite{BS} (see Theorem 1.1', page 394) it follows that
$\sigma^+||u_1||^2_2\leq ||u_1||^2,$
for all $u_1 \in E_1$.
Therefore,
\[
\inf_{u_1 \in E_1, u_1 \not=0}\dfrac{||u_1||^2}{||u_1||^2_2} \geq \sigma^+,
\]
and then, setting
\begin{equation} \label{a_0}
a_0 := \inf_{u_1 \in E_1, u_1 \not=0}\dfrac{||u_1||^2}{||h^{\frac{1}{2}}u_1||^2_2}\geq \dfrac{1}{h_{\infty}}\inf_{u_1 \in E_1, u_1 \not=0}\dfrac{||u_1||^2}{||u_1||^2_2}\geq \dfrac{{\sigma}^+}{h_{\infty}}>0,
\end{equation}
it follows that
\begin{equation}
||u_1||^2\geq a_0\displaystyle\int_{\mathbb{R}^N}h(x)u^2_1(x)dx,
\end{equation}
for all $u_1 \in E_1$.

\qquad Under all the previous assumptions and notations, it is possible to state the first main result of this section.

\begin{theorem}\label{t1} Assume $V$ satisfying $(V_1)-(V_2)$, $h$ satisfying $(h_1)$ and f satisfying $(f_1)-(f_3)$, with $a>a_0$. Then problem $(P)$ has a nontrivial weak solution $u \in H^1(\mathbb{R}^N)$.
\end{theorem}

\qquad In order to  prove Theorem \ref{t1} it is necessary to check that $I$ satisfies hypotheses $(I_1)-(I_4)$ in Theorem \ref{ALT}, and then it is possible to ensure the result by applying this theorem. First of all, see that $I \in C^1(E,\mathbb{R})$, due to the hypotheses assumed about $h$ and $f$.  Moreover, on one hand,
\[
\big(Lu,u\big)= \big(L_1u_1 + L_2u_2, u_1 +u_2\big)= \big(L_1u_1,u_1\big)+\big(L_2u_2,u_2\big),
\]
and on the other hand, denoting by $I_1:E_1\to E_1$ the identity operator in $E_1$, and by $P^-:E_2 \to E_2$ the projector operator of $E_2$ on $E^-$, note that $u_2 \in E_2$ is such that $u_2 = u^-+u^0$, hence
\begin{eqnarray*}
||u^+||^2 - ||u^-||^2= ||u_1||^2 - ||u_2 - u^0||^2 = \big(I_1(u_1),u_1\big)+\big(-P^-(u_2),u_2\big).
\end{eqnarray*}
Thus, setting $L_1 := I_1$ and $L_2 := -P^-$, it follows that $L_i:E_i\to E_i$ are bounded, linear and self-adjoint operators, for $i=1,2.$ Therefore $I(u)= \dfrac{1}{2}\big(Lu,u\big)+B(u)$, where
\begin{equation} \label{B}
B(u)= -\displaystyle\int_{\mathbb{R}^N}h(x)F(u(x))dx,
\end{equation}
and this gives $(I_1).$

\qquad In order to prove $(I_2)$ the following lemma is needed.

\begin{lemma} \label{l2}
	Assume that $(h_1)$ and $(f_1)-(f_2)$ hold for $I$, then $B$ given in (\ref{B}) is weakly continuous and uniformly differentiable on bounded subsets.
\end{lemma}
\begin{proof}
	Let $u_n \rightharpoonup u$ be a \ sequence \ in $E$, \ then $u_n(x) \to u(x)$ almost \ everywhere \ in $\mathbb{R}^N$, \ and $F(u_n)(x) \to F(u)(x)$ almost everywhere in $\mathbb{R}^N$, since $F(s)$ is a continuous function. Moreover, $(\ref{F})$ yields $|F(u_n)|^{\frac{2^*}{p}}) \in L^1(\mathbb{R}^N),$
	since \ $2 < 2\dfrac{2^*}{p}<2^*$ \ and \ $(u_n) \ \subset \ E\hookrightarrow L^s(\mathbb{R}^N),$ \ for \ $2 \leq s \leq 2^*$. Thus \ $\big(F(u_n(\cdot))\big)\subset L^{\frac{2^*}{p}}(\mathbb{R}^N)$ is bounded, provided that $(u_n)$ is bounded in $E$ and then it is bounded in $L^{2\frac{2^*}{p}}(\mathbb{R}^N)$ and in $L^{2^*}(\mathbb{R}^N)$. Since \ $F(u_n)(x)\to F(u)(x)$ \ almost \ everywhere \ in \ $\mathbb{R}^{N}$ and $||F(u_n)||_{L^{\frac{2^*}{p}}(\mathbb{R}^N)} \leq C$ for all $n\in\mathbb{N}$, by Brezis-Lieb's Lemma in \cite{BL}, $F(u_n) \rightharpoonup F(u)$ in $L^{\frac{2^*}{p}}(\mathbb{R}^N)$. Provided that $(h_1)$ implies that $h \in L^q(\mathbb{R}^N)$, where $q$ is the conjugate exponent of $\dfrac{2^*}{p}$, then
	\[
	\int_{\mathbb{R}^N}h(x)F(u_n(x))dx \to \int_{\mathbb{R}^N}h(x)F(u(x))dx,
	\]
	as $n\to +\infty.$ Therefore $B$ is weakly continuous.
	
	\qquad Showing that $B$ is uniformly differentiable on bounded subsets of $E$ means that given $\varepsilon>0$ and $B_R \subset E$, there exists $\delta>0$ such that 
	\[
	|B(u+v)-B(u) - B'(u)v|<\varepsilon||v||,
	\]
	for all $u+v \in B_R$  with $||v||<\delta$. First, note that $B$ satisfies
	\begin{eqnarray}\label{ud1}
	\big|B(u+v) - B(u) -B'(u)v\big|
	&\leq&
	h_\infty^{\frac{1}{2}}\int_{\mathbb{R}^N}(h(x))^{\frac{1}{2}}|f(z(x))-f(u(x))||v(x)|dx,
	\end{eqnarray}
	since for $\psi(t):= F(u+tv)$, it yields $\psi'(t)=f(u+tv)v$. Hence, Mean Value Theorem implies there exists some function $\theta(x)$, such that $0<\theta(x)<1$, almost everywhere in $\mathbb{R}^N$ and writing $z=u+\theta v,$ it follows that
	\[
	F(u+v)-F(u)= \psi(1)-\psi(0)= \psi'(\theta)=f(z)v,
	\]
	almost everywhere in $\mathbb{R}^N$. Moreover, provided that $h \in L^\infty(\mathbb{R}^N)$, from (\ref{ud1})
	\begin{eqnarray}\label{ud2}
	\big|B(u+v) - B(u) -B'(u)v\big|
	&\leq& h_\infty^{\frac{1}{2}}||\xi||_{L^2(\mathbb{R}^N)}\ ||v||_{L^2(\mathbb{R}^N)}\nonumber\\
	&\leq&h_\infty^{\frac{1}{2}}C_{2}||\xi||_{L^2(\mathbb{R}^N)}\||v||,
	\end{eqnarray}
	where $C_{2}>0$ is the constant given by the continuous embedding $E \hookrightarrow L^{2}(\mathbb{R}^N)$ and ${\xi := h^{\frac{1}{2}}(\cdot)|f(z(\cdot))-f(u(\cdot))|}$ it belongs to $L^2(\mathbb{R}^N)$. Indeed, since $\dfrac{2^*}{p}$ is the conjugate exponent of $q$, applying H\"older's Inequality for $q$ and $\dfrac{2^*}{p}$ it follows that
	\begin{eqnarray*}
		\int_{\mathbb{R}^N}|\xi|^2dx &=& \int_{\mathbb{R}^N}h(x)|f(z(x))-f(u(x))|^2dx \\
		&\leq&||h||_{L^q(\mathbb{R}^N)} \ ||f(z(x))-f(u(x))||_{L^{2\frac{2^*}{p}}(\mathbb{R}^N)}^2\\
		&<& +\infty,
	\end{eqnarray*}
	since $2<2\dfrac{2^*}{p}<2^*$ and $|f(u)|^{2\frac{2^*}{p}}\leq \kappa^{2\frac{2^*}{p}}|u|^{2\frac{2^*}{p}} \in L^1(\mathbb{R}^N)$  due to assumption $(h_1)$. Observe that by (\ref{ud2}) is sufficient to show that given $\varepsilon>0$ and $B_R\subset E$, there exists $\delta>0$ such that $||\xi||_{L^2(\mathbb{R}^N)} \leq \dfrac{\varepsilon}{h_\infty^{\frac{1}{2}}C_2}$ for all $u+v \in B_R$ with $||v||<\delta$.
	
	\qquad In \ order \ to \ prove this \ indirectly, \ suppose \ there \ exists $\varepsilon_0>0$ and $
	B_{R_0}\subset E$ fixed, such that for all $\delta>0$ it is possible to obtain $u_\delta + v_\delta \in B_{R_0}$ with $||v_\delta||< \delta$ and $||\xi_\delta||_{L^2(\mathbb{R}^N)}>\dfrac{\varepsilon_0}{h_\infty^{\frac{1}{2}}C_2}$, where
	\[
	\xi_\delta = h^{\frac{1}{2}}(\cdot)|f(z_\delta(\cdot))-f(u_\delta(\cdot))| \text{ \ and \ } z_\delta = u_\delta + \theta v_\delta.
	\]
	Choosing $\delta_n = \dfrac{1}{n}$, for each $n \in \mathbb{N}$ there exist $u_n +v_n \in B_{R_0}$ such that $||v_n||\leq \dfrac{1}{n}$ and 
	\[
	||\xi_n||_{L^2(\mathbb{R}^N)}>\dfrac{\varepsilon_0}{h_\infty^{\frac{1}{2}}C_2}.
	\]
	Hence, $v_n \to 0$ in $E$ and $u_n \rightharpoonup u$ in $E$, up to subsequences as $n \to +\infty$. In addition, $z_n \rightharpoonup u$ in $E$ and $z_n(x), u_n(x)\to u(x)$ almost everywhere in $\mathbb{R}^N$, up to subsequences. Thus,
	\[
	|f(z_n(x))-f(u_n(x))|^2\to 0,
	\]
	almost everywhere in $\mathbb{R}^N$ as $n \to +\infty$. Moreover, $(z_n)$ and $(u_n)$ are bounded in $E$ and the Sobolev embedding $L^{2\frac{2^*}{p}}(\mathbb{R}^N) \hookrightarrow E$ holds, then
	\begin{eqnarray*}
		||\ |f(z_n)-f(u_n)|^2||_{L^{\frac{2^*}{p}}(\mathbb{R}^N)}^{\frac{2^*}{p}}
		&\leq& C\left(||f(z_n)||_{L^{2\frac{2^*}{p}}(\mathbb{R}^N)}^{2\frac{2^*}{p}} + ||f(u_n)||_{L^{2\frac{2^*}{p}}(\mathbb{R}^N)}^{2\frac{2^*}{p}}\right) \\
		&\leq& C\kappa^{2\frac{2^*}{p}}\left(||z_n||_{L^{2\frac{2^*}{p}}(\mathbb{R}^N)}^{2\frac{2^*}{p}}+||u_n||_{L^{2\frac{2^*}{p}}(\mathbb{R}^N)}^{2\frac{2^*}{p}}\right)\\
			&\leq&2C(C_{2\frac{2^*}{p}}\kappa R_0)^{2\frac{2^*}{p}},
	\end{eqnarray*}
	where $C_{2\frac{2^*}{p}}>0$ is the constant given by the continuous embedding $E \hookrightarrow L^{2\frac{2^*}{p}}(\mathbb{R}^N)$.
	Therefore, the sequence $(|f(z_n)-f(u_n)|^2)$ is bounded in ${L^{\frac{2^*}{p}}(\mathbb{R}^N)}$. Applying Brezis-Lieb's Lemma again, it yields $|f(z_n)-f(u_n)|^2 \rightharpoonup 0$ in ${L^{\frac{2^*}{p}}(\mathbb{R}^N)}$ as $n \to +\infty$. Since $h\in L^q(\mathbb{R}^N)$, which is the dual space of ${L^{\frac{2^*}{p}}(\mathbb{R}^N)}$, by weak convergence it yields
	\begin{eqnarray*}
		||\xi_n||^2_{L^{2}(\mathbb{R}^N)} &=& \int_{\mathbb{R}^N}h(x)|f(z_n(x))-f(u_n(x))|^2dx \to 0,
	\end{eqnarray*}
	as $n \to +\infty,$ which contradicts $||\xi_n||_{L^2(\mathbb{R}^N)}>\dfrac{\varepsilon_0}{h_\infty^{\frac{1}{2}}C_2}$ and completes the proof.
\end{proof}

\qquad In order to prove $(I_3)$, choose $Q= \{re:r \in[0,r_1]\}\oplus(E_2\cap B_{r_2})$, and $S=\partial B_{\rho}\cap E_1$, where $0 < \rho < r_1 <r_2 $ are constants and $e \in E_1$, $||e||=1,$  must be a suitable vector. Hence, observe that if $a$ as in $(f_2)$ is such that $a>a_0$, then for $\varepsilon>0$ small enough and $a_\varepsilon:=a-\varepsilon$, it follows that $a> a_{\varepsilon} >a_0$, and by the definition of $a_0$ in  (\ref{a_0}), there exists some $e_0 \in E_1$ such that
\[
a_0 \displaystyle\int_{\mathbb{R}^N}h(x)e^2_0(x)dx\leq||e_0||^2\leq a_{\varepsilon}\displaystyle\int_{\mathbb{R}^N}h(x)e^2_0(x)dx.
\]
Normalizing $e_0$ it follows that $e= \dfrac{e_0}{||e_0||}\in E_1$ is such that
\begin{equation}\label{e}
1 = ||e||^2 = \int_{\mathbb{R}^N}\Big(|\nabla e(x)|^2 + V(x)e^2(x)\Big)dx \leq a_{\varepsilon} \int_{\mathbb{R}^N}h(x)e^2(x)dx.
\end{equation}
Therefore, choose such $e$ for the structure of $Q$. Furthermore, by Lemma \ref{LE} 
\ $S$ and $\partial Q$ \ ``link'' , \ where $\partial Q$  can be written as\ $\partial Q=Q_1\cap Q_2 \cap Q_3$, with \  $Q_1= \{0\}\oplus(E_2\cap B_{r_2}),$ \ \  $Q_2=\{re:r\in[0,r_1]\}\oplus(E_2\cap\partial B_{r_2})$ and $Q_3= \{r_1e\}\oplus(E_2\cap B_{r_2})$.
The following lemma shows that $I$ satisfies $(I_3) \ (i)-(ii)$ in Theorem \ref{ALT} for some $\alpha >0$, $\omega = 0$, and arbitrary $v\in E_2$.
\begin{lemma} \label{l3}
	Assume that $(V_1)-(V_2)$, $(h_1)$ and $(f_1)-(f_2)$ hold for $I$. For $Q$ and $S$ as above, and for sufficiently large $r_1>0$, $I|_S\geq \alpha >0$ and $I|_{\partial Q}\leq 0$ hold, for some $\alpha>0.$
\end{lemma}
\begin{proof} By definition, $S\subset E_1$, hence for all $u_1 \in S$ it yields
	\begin{eqnarray}\label{l3e1}
	I(u_1)	&\geq&\dfrac{1}{2}\rho^2- h_{\infty}\int_{\mathbb{R}^N}\left(\dfrac{\varepsilon}{2}|u_1(x)|^2 + \dfrac{C_{\varepsilon}}{p}|u_1(x)|^p\right)dx \nonumber\\
		&=& \rho^2\left(\dfrac{1}{2}\big(1-\varepsilon h_{\infty}C_2^2\big) - \dfrac{C_{\varepsilon}}{p}h_{\infty}C_p^p\rho^{p-2}\right).
	\end{eqnarray}
	Thus, if $\varepsilon, \rho$ are sufficiently small, from (\ref{l3e1}) it holds $I(u_1) \geq \alpha>0.$
	
	\qquad Now, for the purpose of checking that $I|_{\partial Q}\leq 0 < \alpha, $ consider the three cases as follows:\\
	\textbf{Case i.} \ $ u \in Q_1\subset E_2$, thus
	\[
	I(u)= -\dfrac{1}{2}||u||^2-\int_{\mathbb{R}^N}h(x)F(u(x))dx \leq0,
	\]
	since $h(x)F(u(x))\geq0$, for all $x \in \mathbb{R}^N.$\\
	\textbf{Case ii.} \ $u \in Q_2$, thus $u=u_1+u_2,$ where $u_1= re$, with $0\leq ||u_1||= r \leq r_1$ and $||u_2||=r_2>r_1$, therefore
	\[
	I(u)= \dfrac{1}{2}\left(||u_1||^2-r_2^2\right)-\int_{\mathbb{R}^N}h(x)F(u(x))dx \leq \dfrac{1}{2}\left(r_1^2-r_2^2\right)<0.
	\]
	\textbf{Case iii.} \ $u \in Q_3$, thus $u=r_1e+u_2$, where $0\leq ||u_2||\leq r_2.$ If $r_1\leq ||u_2||\leq r_2$, then
	\[
	I(u)=\dfrac{1}{2}\left(r_1^2-||u_2||^2\right)-\int_{\mathbb{R}^N}h(x)F(u(x))dx \leq \dfrac{1}{2}\left(r_1^2-r_1^2\right)\leq0.
	\]
	If $0\leq||u_2||<r_1$, put $u_2=r_1v_2$, where $v_2 \in B_1\cap E_2.$ Thus,
	\begin{eqnarray}\label{c3e1}
	I(u)&=& \dfrac{1}{2}r_1^2\left(1-||v_2||^2\right)-\int_{\mathbb{R}^N}h(x)F(u(x))dx \nonumber\\
	&\leq& \dfrac{1}{2}r_1^2\left(1 - \int_{\mathbb{R}^N}2h(x)\dfrac{F\big(r_1(e(x)+v_2(x))\big)}{r_1^2}dx\right).
	\end{eqnarray}
	\textbf{Claim.} \textit{ The limit
		\[\lim_{r_1 \to +\infty}\int_{\mathbb{R}^N}2h(x)\dfrac{F\big(r_1(e(x)+v_2(x))\big)}{r_1^2}dx = a\int_{\mathbb{R}^N}h(x)\big[e(x)+ v_2(x)\big]^2dx,
		\]
		is uniform for $v_2 \in B_1\cap E_2$.}\\
	
	\qquad Assume that claim postponing its proof, in order to conclude case $iii$. From uniform convergence in $B_1 \cap E_2$, for each $\varepsilon>0$ there exists $r_0>0$ such that, for all $r_1\geq r_0$
		\[
		\int_{\mathbb{R}^N}\left[a-2\dfrac{F\big(r_1(e(x)+v_2(x))\big)}{r_1^2(e(x)+v_2(x))^2}\right]h(x)\big(e(x)+v_2(x)\big)^2dx< \varepsilon\int_{\mathbb{R}^N}h(x)\big(e(x)+v_2(x)\big)^2dx,
		\]
	for all $v_2 \in B_1 \cap E_2.$ Thus,
	\[
	\dfrac{1}{2}r_1^2\left(1 - \int_{\mathbb{R}^N}2h(x)\dfrac{F\big(r_1(e(x)+v_2(x))\big)}{r_1^2}dx\right)< \dfrac{1}{2}r_1^2\left(1 - a_{\varepsilon}\int_{\mathbb{R}^N}h(x)e^2(x)dx\right),
	\]
	where  $a_{\varepsilon}=(a-\varepsilon)$. Substituting in (\ref{c3e1}), it yields
	\begin{equation}\label{c3e3}
	I(u)<\dfrac{1}{2}r_1^2\left(1 - a_{\varepsilon}\int_{\mathbb{R}^N}h(x)e^2(x)dx\right).
	\end{equation}
	Therefore, from (\ref{e}) and (\ref{c3e3}) it follows that
	$I(u)<0$, and the result holds.
	
	\qquad Now, we prove the claim. In order to do so, define for all $n \in \mathbb{N}$, the functional $J_n: B_1\cap E_2 \to \mathbb{R}$ given by
	\[
	J_n(v_2):= \int_{\mathbb{R}^N}\left[a-2\dfrac{F\big(n(e(x)+v_2(x))\big)}{n^2(e(x)+v_2(x))^2}\right]h(x)\big(e(x)+v_2(x)\big)^2dx.
	\]
	The continuity of $F$ implies that $J_n$ is continuous for all $n \in \mathbb{N}$. From $(f_2)$ and by the equivalence of $H^1(\mathbb{R}^N)$ and $E$ norms,
	\[
	0\leq J_n(v_2)\leq ah_{\infty}\Big(||e||^2_2 + ||v_2||^2_2\Big)\leq 2C_2^2ah_{\infty},
	\]
	for all $v_2 \in B_1 \cap E_2$. Then, seeing that $E_2$ is finite dimensional$, B_1 \cap E_2$ is compact, and since $J_n$ is continuous in $B_1 \cap E_2$, it attains a maximum value, denoted by $u_n \in B_1 \cap E_2$. Considering this sequence of maximums $(u_n)$, and provided that $||u_n||\leq 1$ for all $n \in\mathbb{N}$, the sequence is bounded. Again, provided that $E_2$ is finite dimensional such a sequence converges, up to subsequences, in $B_1\cap E_2$, namely, $u_n \to u$ in $E-$norm. Moreover, for all $v_2 \in B_1 \cap E_2$, for all $n \in \mathbb{N}$, $0\leq J_n(v_2)\leq J_n(u_n)$ holds, that is
	\begin{eqnarray}\label{l2e1}
	0&\leq& \int_{\mathbb{R}^N}\left[a-2\dfrac{F\big(n(e(x)+v_2(x))\big)}{n^2(e(x)+v_2(x))^2}\right]h(x)\big(e(x)+v_2(x)\big)^2dx \nonumber\\ &\leq& \int_{\mathbb{R}^N}\left[a-2\dfrac{F\big(n(e(x)+u_n(x))\big)}{n^2(e(x)+u_n(x))^2}\right]h(x)\big(e(x)+u_n(x)\big)^2dx.
	\end{eqnarray}
	Now, note that $u_n(x)\to u(x)$ almost everywhere in $\mathbb{R}^N$, then from $(f_2)$
	\[
	\left[a-2\dfrac{F\big(n(e(x)+u_n(x))\big)}{n^2(e(x)+u_n(x))^2}\right]h(x)\big(e(x)+u_n(x)\big)^2 \to 0,
	\]
	almost everywhere in $\mathbb{R}^N$ as $n \to +\infty$. More than this, since $u_n \to u$ in $E$, $u_n \to u$ in $L^2(\mathbb{R}^N)$,  and so by Lebesgue Dominated Convergence Theorem, it follows
	\[
	\lim_{n \to +\infty}\int_{\mathbb{R}^N}\left[a-2\dfrac{F\big(n(e(x)+u_n(x))\big)}{n^2(e(x)+u_n(x))^2}\right]h(x)\big(e(x)+u_n(x)\big)^2dx=0,
	\]
	Now, applying in (\ref{l2e1}) the limit as $n \to +\infty$, the Sandwich Theorem  yields
	\[
	\lim_{n \to +\infty}\int_{\mathbb{R}^N}\left[a-2\dfrac{F\big(n(e(x)+v_2(x))\big)}{n^2(e(x)+v_2(x))^2}\right]h(x)\big(e(x)+v_2(x)\big)^2dx=0,
	\]
	uniformly for all $v_2 \in B_1\cap E_2$ and the claim is proved.
\end{proof}
\qquad For now, observe that by $(f_1)$ and $(f_2)$, given $\varepsilon>0$ and $2< p < 2^*$ there exists a constant $C_{\varepsilon}>0$ such that 
\begin{equation}\label{F}
|F(s)| \leq \dfrac{\varepsilon}{2}|s|^2
+ \dfrac{C_{\varepsilon}}{p}|s|^p
\end{equation}
and
\begin{equation}\label{f}
|f(s)|\leq\varepsilon|s|+ C_{\varepsilon}|s|^{p-1}
\end{equation}
for all $s \in \mathbb{R}.$

\qquad In order to \ verify \ $(I_4)$ it is necessary to ensure the boundedness of Cerami sequences for $I$. Next lemma gives this result.

\begin{lemma}\label{l1}
	Assume that $(V_1)-(V_2), (h_1)$ and $(f_1)-(f_3)$ hold for $I$ and let  $(u_n) \subset E$ be a Cerami sequence of $I$ on an arbitrary level $c \in \mathbb{R}$, then $(u_n)$ is bounded.
\end{lemma}

\begin{proof}
	Suppose by contradiction that $||u_n||\to +\infty$ as $n \to +\infty$, up to subsequences. Defining $v_n:=\dfrac{u_n}{||u_n||}$ it follows that $(v_n)$ is a bounded sequence in $E$. Then $v_n \rightharpoonup v$ as $n \to +\infty,$  up to subsequences. Let us show that neither $v=0$, nor $v\not=0$ can occur.
	
	\qquad First, suppose that $v\not=0$, it means there exists $\Omega \subset \mathbb{R}^N$ such that $|\Omega|>0$ and $v(x)\not=0$ for all $x \in \Omega$. Since $v_n(x) \to v(x)$ almost everywhere in $\Omega$, one conclude that $|u_n(x)|\to +\infty$,  almost everywhere for $x \in \Omega$. Hence, in view of $(f_3)$ one arrives at $h(x)Q(u_n(x)) \to +\infty$ as $n \to +\infty$, almost everywhere in $\Omega$. Applying Fatou's Lemma, one obtains
\begin{equation}\label{be1}
	\liminf_{n \to +\infty}\int_{\Omega}h(x)Q(u_n(x))\;dx \geq \int_{\Omega}	\liminf_{n \to +\infty}h(x)Q(u_n(x))\;dx = +\infty.
\end{equation}

	\qquad Provided that $(u_n)$ is a Cerami sequence on level $c$, it follows that
	\begin{equation}\label{be2}
c + o_n(1)  = I(u_n) - \dfrac{1}{2}I'(u_n)u_n = \int_{\mathbb{R}^N}h(x)Q(u_n(x))\;dx \geq \int_{\Omega}h(x)Q(u_n(x))\;dx.
	\end{equation}
Combining (\ref{be1}) and (\ref{be2}) it yields a contradiction. Therefore, one must have $v=0$.
	
	\qquad Setting $v_n = v_{+,n} + v_{-,n}+ v_{0,n}$, where $v_{i,n} \in E^j, j= +,-,0$, up to subsequences it yields
	\begin{eqnarray}\label{be3}
	{v}_n\rightharpoonup{v} = {v}^+ + {v}^- + {v}^0 &\text{ \ in \ }& E= E^++E^-+E^0,\nonumber\\ {v}_n \to {v} &\text{ \ in \ }& L^2_{loc}(\mathbb{R}^N).
	\end{eqnarray}
Since $v = 0$, hence $v^+ = v^- = v^0=0$ and from (\ref{be3}) one has $v_{j,n}(x) \to 0$ almost everywhere in $\mathbb{R}^N$, moreover, $v_{0,n}\to 0$ in $E$, provided that $E^0$ is finite dimensional. In addition, $(u_n)$ is a Cerami sequence, hence $I'(u_n)u_{+,n} \to 0$ and  $I'(u_n)u_{-,n} \to 0$ as $n \to +\infty.$ Therefore,

\begin{eqnarray*}
	o_n(1)&=& \dfrac{I'(u_n)u_{+,n}}{||u_n||^2} - \dfrac{I'(u_n)u_{-,n}}{||u_n||^2}\nonumber\\
	&=&
	||v_{+,n}||^2+||v_{-,n}||^2 - \int_{\mathbb{R}^N}h(x)\left[\dfrac{f(u_n(x))}{u_n(x)}v_n(x)\Big(v_{+,n}(x)-v_{-,n}(x)\Big)\right]dx\\
	&=&1- ||v_{0,n}||^2 -\int_{\mathbb{R}^N}h(x)\left[\dfrac{f(u_n(x))}{u_n(x)}\Big(v_{+,n}^2(x)-v_{-,n}^2(x)\Big)\right]dx,
\end{eqnarray*}
which implies that, 
\begin{equation} \label{be4}
\int_{\mathbb{R}^N}h(x)\left[\dfrac{f(u_n(x))}{u_n(x)}\Big(v_{+,n}^2(x)-v_{-,n}^2(x)\Big)\right]dx \to 1 - ||v^{0}||^2 = 1,
\end{equation}
as \ $n \to +\infty.$ \ However, \ since \ $\left|\dfrac{f(s)}{s}\right|\leq \kappa$, for \ all \ $s \in \mathbb{R}\backslash \{0\}$ \ and \ provided \ that $h \in L^{q}(\mathbb{R}^N)$, \ with \ ${q =\frac{2^*}{2^*-p}}$ \ and \ $2<2\dfrac{2^*}{p}<2^*$, Holder's \ inequality \ ensures \ that ${\phi_n(x):= \dfrac{f(u_n(x))}{u_n(x)}\Big(v_{+,n}^2(x)-v_{-,n}^2(x)\Big) \in L^{^\frac{2^*}{p}}(\mathbb{R}^N)}$ for all $n$ and $(\phi_n)$ is a bounded sequence in $L^{^\frac{2^*}{p}}(\mathbb{R}^N)$, since $(v_n)$ is bounded in $E$. Furthermore, in view of (\ref{be3}) one has $\phi_n(x)\to 0$ almost everywhere in $\mathbb{R}^N$, thus, it yields that $(\phi_n)$ converges weakly to $0$ in $L^{^\frac{2^*}{p}}(\mathbb{R}^N)$, up to subsequences. 

\qquad Recalling that $h \in L^{q}(\mathbb{R}^N)$, with $q =\frac{2^*}{2^*-p}$, the weak convergence implies that
\begin{equation}\label{be5}
\int_{\mathbb{R}^N}h(x)\left[\dfrac{f(u_n(x))}{u_n(x)}\Big(v_{+,n}^2(x)-v_{-,n}^2(x)\Big)\right]dx \to 0,
\end{equation}
as $n \to +\infty$. Looking at (\ref{be4}) and (\ref{be5}), one arrives at a contradiction.
\end{proof}

 \qquad In view of last result, to obtain $(I_4)$ we fix \  $b>0$ \ and \  take \ $(u_n)$  \ such \ that \ $I(u_n) \subset [c-b,c+b]$ \ and $||I'(u_n)||\big(1+||u_n||\big) \to 0$ as $n\to +\infty.$ Supposing that $(u_n)$ is unbounded we take $(u_{n_k})\subset (u_n)$ such that $||u_{n_k}||\to +\infty$ as $k\to\infty.$ Seeing that $I(u_{n_k})\subset[c-b,c+b]$ is bounded in $\mathbb{R}$, it implies that $I(u_{n_k}) \to d$, up to subsequences. Then, $(u_{n_k})$ is a Cerami sequence on level $d$, up to subsequences, hence $(u_{n_k})$ is bounded up to subsequences, by Lemma \ref{l1}. However, it yields a contradiction, since $||u_{n_k}||\to +\infty$ as $k \to +\infty.$ Thus, $(u_n)$ is bounded.

\qquad Now, after all theses results, one is finally ready to prove Theorem \ref{t1}.\\

\begin{proof}[Proof of Theorem \ref{t1}.]
In view of all assumptions $I$ satisfies $(I_1),(I_2),(I_3)$ and $(I_4)$ in Theorem \ref{ALT}, so it is possible to apply  this theorem for $I$. Theorem \ref{ALT} provides  a $c\geq \alpha >\omega=0,$ critical value of $I$. Therefore, there exists $u \in E$ such that $I(u)=c>0$ and $I'(u)=0$, hence, $u\not=0$, since $I(u)>0$. Provided that $I\in C^1(E,\mathbb{R})$, it follows that $u$ is a nontrivial weak solution to $(P)$ in $H^1(\mathbb{R}^N).$ 
\end{proof}

\begin{Remark} \label{s2r3}
	Note that problem $(P)$ in (\ref{P}), has just been solved with conditions on the potential $V$, which ensures a spectrum with negative and positive parts. Such conditions implied that the subspace $E^-$, corresponding to negative spectrum, was finite dimensional (see Remark \ref{s2r1})). Although the fact of $E_2=E^-\oplus E^0$ being finite dimensional was not necessary to apply Theorem \ref{t1}, this information was used to obtain the linking geometry (see the proof of claim in Lemma \ref{l3}). However, with minor changes it is possible to prove the linking geometry indirectly, without the assumption that $\dim E_2$ is finite.
\end{Remark}
\qquad Since Theorem \ref{ALT} does not require that any subspace in the linking decomposition needs to be finite dimensional, the main goal now  is to work with the same problem, but assuming conditions on $V$ which gives both subspaces in the linking decomposition being infinite dimensional. Henceforth, consider problem $(P)$, but replacing the condition $(V_1)$ on $V$ by the condition $(V_1')$,
and also assuming the condition $(V_2).$ The assumptions on $h$ and $f$ are the same as before.

\begin{Remark}\label{s2r4}
	In view of  $(V_1')$,  $V$ is  periodic and continuous,  hence bounded. In addition,  by \cite{RS4} (see page 309, Theorem XIII.100) 
	the  spectral theory 
	asserts  that operator $A$ has pure  absolutely  continuous  spectrum,  which is  bounded from below  and consists of  closed  disjoint 
	intervals.  Namely, $\sigma(A)=\sigma_{ess}(A)=\sigma_{ac}(A)=\cup[\alpha_i,\beta_i]$. In view of $(V_2)$, operator $A$ has nonempty negative and positive spectra, and $0$ is in the gap between them. Indeed, if $0\in\sigma(A)$ it would be an isolated point of the spectrum, which contradicts the fact that operator $A$ has pure continuous spectrum. Moreover, since  $\sigma(A)=\sigma_{ess}(A)$ the negative and positive spectrum are both part of essential spectrum. Thus, if $\{\mathcal{E}(\lambda)\}$ is the spectral family of operator $A$, by spectral theory and essential spectrum definition (see \cite{BS} and \cite{RS4}), the subspaces associated with negative and positive spectrum, namely $\mathcal{E}(0)(H^1(\mathbb{R}^N))$ and $(I-\mathcal{E}(0))(H^1(\mathbb{R}^N))$ are both infinite dimensional.
\end{Remark}

\hspace{0,75cm}Setting $E_1:=(I-\mathcal{E}(0))(H^1(\mathbb{R}^N)) $ and $E_2:= \mathcal{E}(0)(H^1(\mathbb{R}^N))$, these subspaces are such that operator $A$ is positive definite in the former and negative definite in the latter. Indeed, by the spectral family definition (see \cite{BS}, Theorem 1.1', page 394; see also \cite{P} Chapter 3),
for all $u_1 \in E_1$  and  for all $u_2 \in E_2$ it yields 
\[
\sigma^+||u_1||^2_2\leq \int_{\mathbb{R}^N}\Big(|\nabla u_1(x)|^2+V(x)u^2_1(x)\Big)dx
\]
and
\[
-\sigma^-||u_2||^2_2\leq -\int_{\mathbb{R}^N}\Big(|\nabla u_2(x)|^2+V(x)u^2_2(x)\Big)dx.
\]
Because of this, it is possible to proceed similarly to before and consider the norm induced by operator $A$:
\[
||u||^2 = \left\{
\begin{array}{lll}
\ \ \ \displaystyle\int_{\mathbb{R}^N}\Big(|\nabla u|^2 + V(x)u^2\Big)dx = (Au,u)_2, \ \ \ \ \ \text{if} \ \ u \in E_1, \\
\\
-\displaystyle\int_{\mathbb{R}^N}\Big(|\nabla u|^2 + V(x)u^2\Big)dx=-(Au,u)_2, \ \ \ \text{if} \ \ u \in E_2,
\end{array}
\right.
\]
which is equivalent to the usual norm in $H^1(\mathbb{R}^N)$. Denote  $E=(H^1(\mathbb{R}^N), ||\cdot||)$, then $E=E_1\oplus E_2$, namely every $u \in E$ can be uniquely written as $u=u_1+u_2$, with $u_i \in E_i$ and
\[
||u||^2 = ||u_1||^2+||u_2||^2 = (Au_1,u_1) - (Au_2,u_2).
\]
Here the conclusions from Remark \ref{s2r2} also hold, and again the functional $I:E\to \mathbb{R}$ associated with $(P)$ is written as
\[
I(u)= \dfrac{1}{2}\Big(||u_1||^2 - ||u_2||^2\Big) - \int_{\mathbb{R}^N}h(x)F(u(x))dx,
\]
for all $u=u_1 + u_2 \in E.$\\

\qquad Now it is possible to state the second main result of this section.

\begin{theorem}\label{t2}
	Let consider problem $(P)$ with $V$ satisfying $(V_1')-(V_2)$, $h$ satisfying \ $(h_1)$ \ and $f$ satisfying $(f_1)-(f_3)$, \ with $a>a_0$. \ Then $(P)$ has \ a \ nontrivial \ weak \ solution \ $u \in H^1(\mathbb{R}^N)$.
\end{theorem}	

\qquad Under the purpose of applying Theorem \ref{ALT} to solve this problem, proceeding as before it is necessary to show that this problem satisfies all required assumptions of the Abstract Linking Theorem. 

\qquad Observe that as before $I \in C^1(E,\mathbb{R})$, due to the hypotheses assumed about $h$ and $f$.  Moreover, on one hand,
\[
\big(Lu,u\big)= \big(L_1u_1 + L_2u_2, u_1 +u_2\big)= \big(L_1u_1,u_1\big)+\big(L_2u_2,u_2\big),
\]
and on the other hand, denoting by $I_i:E_i\to E_i$ the identity operator in $E_i$ for $i=1,2$, note that 

\[||u_1||^2 - ||u_2||^2 = \big(u_1,u_1\big) - \big(u_2,u_2\big) = \big(I_1(u_1),u_1\big)+\big(-I_2(u_2),u_2\big).
\]
Thus, setting $L_i := I_i$ for $i=1,2$, it follows that $L_i:E_i\to E_i$ are bounded, linear and self-adjoint operators, for $i=1,2$. Thus, $I$ satisfies $(I_1)$ in Theorem \ref{ALT} as before. 

\qquad Furthermore, all assumptions are kept on $h$ and $f$, then $(I_2)$ and $(I_4)$ also hold here with the same proofs as before. Therefore, it is only necessary to show the linking geometry in $(I_3)$, which will have a different proof in this case, provided that $E_2$ is infinite dimensional. First of all, set $S:= \partial B_{\rho}\cap E_1$ and ${Q= \{re+u_2: r\geq0, u_2 \in E_2, ||re+u_2||\leq r_1\}}$, where $0<\rho < r_1$ are constants and $e\in E_1, ||e||=1$, is chosen as before. In fact, assuming $a>a_0$, by the definition of $a_0$, there exists some unitary $e \in E_1$ such that
\begin{equation}\label{l5e2}
a_0 \displaystyle\int_{\mathbb{R}^N}h(x)e^2(x)dx\leq||e||^2 =1 < a\displaystyle\int_{\mathbb{R}^N}h(x)e^2(x)dx.
\end{equation}
Such $e$ makes the following lemma true. Moreover, it is possible to show that $S$ and $Q$ \ ``link'' following the same lines as in Lemma \ref{LE}. Next lemma shows the linking geometry $(I_3) \ (i)-(ii)$ of Theorem \ref{ALT}, for some $\alpha>0$, $ \omega=0$ and arbitrary $v \in E_2$.

\begin{lemma}\label{l5}
	Assume that $(V_1')-(V_2)$, $(h_1)$ and $(f_1)-(f_2)$ hold for $I$. For $Q$ and $S$ as above, and for sufficiently large $r_1>0$, it follows that $I|_S\geq \alpha >0$ and $I|_{\partial Q}\leq 0$, for some $\alpha>0.$ 
\end{lemma}
\begin{proof}
	The proof that $I|_S\geq \alpha >0$, for some $\alpha>0$, is the same in Lemma \ref{l3}, thus it is not repeated here. In purpose of proving that  $I|_{\partial Q}\leq 0$, observe that $I(u_2)\leq 0$, for all $u_2 \in E_2$, then it suffices to show that $I(re+u)\leq 0$ for $r>0, u \in E_2$ and $||re+u||\geq r_1$, for some $r_1>0$ large enough. Arguing indirectly assume that for some sequence $(r_ne+u_n)\subset \mathbb{R}^+e \oplus E_2$ such that $ ||r_ne+u_n||\to +\infty$, $I(r_ne+u_n)>0$ holds for all $n \in \mathbb{N}$. Then, seeking a contradiction, the desired result holds. Firstly, set $v_n := \dfrac{r_ne+u_n}{||r_ne+u_n||}= s_ne+ w_n$, where $ s_n \in \mathbb{R}^+, w_n \in E_2$ and $||v_n||=1$. Provided that $(v_n)$ is bounded, up to subsequences it follows that $v_n \rightharpoonup v= se + w$ in $E$. Then, $v_n(x) \to v(x)$ almost everywhere in $\mathbb{R}^N$, and seeing that $1 = ||s_ne+w_n||^2= s_n^2 +||w_n||^2$, it ensures that $0\leq s^2_n\leq 1$, $w_n \rightharpoonup w$ in $E$ and $s_n \to s$ in $\mathbb{R}^+.$  Then it yields
	\begin{eqnarray}\label{l5e1}
	\dfrac{I(r_ne+u_n)}{||r_ne+u_n||^2} &=& s^2_n - \dfrac{1}{2} - \int_{\mathbb{R}^N}h(x)\dfrac{F(r_ne(x)+u_n(x))}{||r_ne+u_n||^2}dx  >0,
	\end{eqnarray}
	hence $0<s\leq1$. 
	Moreover, from (\ref{l5e2}) there exists a bounded domain $\Omega_0\subset \mathbb{R}^N$ such that
	\[
	1< a\displaystyle\int_{\Omega_0}h(x)e^2(x)dx.
	\]
	Hence, $supp (e) \cap \Omega_0 \not= \emptyset$, and it follows that
	\begin{eqnarray}\label{l5e3}
	0&>&s^2-s^2a\displaystyle\int_{\Omega_0}h(x)e^2(x)dx\nonumber\\
	&\geq& 
	s^2\left(2 - a\displaystyle\int_{\Omega_0}h(x)e^2(x)dx\right) - 1 - a\displaystyle\int_{\Omega_0}h(x)w^2(x)dx.
	\end{eqnarray}
	On the other hand, since $v_n \rightharpoonup v$ in $E$, it converges strongly in $L^2(\Omega_0)$,
	and since $||r_ne + u_n||\to +\infty$ as $n\to +\infty$, in view of $(f_2)$ it follows that
	\[
	h(x)\dfrac{F(r_ne(x)+u_n(x))}{||r_ne+u_n||^2} = h(x)\dfrac{F(v_n(x)||r_ne(x)+u_n(x)||)v^2_n(x)}{v_n(x)^2||r_ne+u_n||^2} \to h(x)\dfrac{a}{2}v^2(x),
	\]
	almost everywhere in $\Omega_0\cap supp(v)$, which is not empty since $v= se+w$, $(e,w)_2=0$ and $supp(e)\cap\Omega_0 \not= \emptyset$. Thus, by Lebesgue Dominated Convergence Theorem, 
	\[
	\int_{\Omega_0}h(x)\dfrac{F(r_ne(x)+u_n(x))}{||r_ne+u_n||^2}dx \to \dfrac{a}{2}\int_{\Omega_0}h(x)\left(s^2e^2(x) + w^2(x)\right)dx,
	\]
	as $n \to +\infty.$ From (\ref{l5e1}) \begin{equation}\label{l5e4}
	0<2s^2_n - 1 - 2\int_{\Omega_0}h(x)\dfrac{F(r_ne(x)+u_n(x))}{||r_ne+u_n||^2}dx,
	\end{equation}
	and passing to the limit in (\ref{l5e4}) as $n \to + \infty$, it yields
	\begin{eqnarray}\label{l5e5}
	0&\leq& 2s^2 -{1} - {a}\int_{\Omega_0}h(x)\left(s^2e^2(x) + w^2(x)\right)dx \nonumber\\ &=& s^2\left(2 - a\displaystyle\int_{\Omega_0}h(x)e^2(x)dx\right) - {1} - a\displaystyle\int_{\Omega_0}h(x)w^2(x)dx,
	\end{eqnarray}
	which is contrary to (\ref{l5e3}). Therefore the lemma is proved.
\end{proof}

\qquad By Lemma \ref{l5}, the functional $I$ satisfies $(I_3)$ of Theorem \ref{ALT}. Now, Theorem \ref{t2} can be proved. \\

\hspace{-0.5cm}\textit{Proof of Theorem \ref{t2}.} Due to all hypotheses on $I$, it satisfies $(I_1),\ (I_2),\ (I_3)$ and $(I_4)$ in Theorem \ref{ALT}, then it is possible to apply this theorem for $I$. By Theorem \ref{ALT}, $c\geq \alpha >0,$ is a critical value of $I$. Therefore, there exists $u \in E$, such that $I(u)=c>0$ and $I'(u)=0$, provided that $I(u)>0,$ then $u\not=0$. Since $I\in C^1(E,\mathbb{R})$, $u$ is a nontrivial weak solution to $(P)$ in $H^1(\mathbb{R}^N).$	
\vspace{-0.4cm}
\begin{flushright}$\square$\end{flushright}


\bigskip

\begin{thebibliography}{99}                                                                                            
\bibitem{BBF} Bartolo, P., Benci, V. and Fortunato, D.: Abstract Critical Point Theorems and Applications to some Nonlinear Problems with "Strong" Resonance at Infinity. \textit{Nonlinear Analysis Theory, Methods \& Applications} \text{7}, 981-1012 (1983).


\bibitem{BD} Bartsch, T. and Ding, Y.H.: Deformation Theorems on Non-metrizable Vector Spaces and
Applications to Critical Point Theory. \emph{Math. Nachr.} \textbf{279}, 1267-1288 (2006).

\bibitem{BR} Benci, V. and Rabinowitz, P. H.: Critical Point Theorems for Indefinite Functionals. \emph{ Inventiones Math.} \textbf{52}, 241-273 (1979). 

\bibitem{BS} Berezin, F. A. and Shubin, M. A.: \textit{The Schorodiger Equation}, Kluwer Academic Publishers (1991).

\bibitem{BL} Brezis, H. and Lieb, E. H.: A relation between pointwise convergence of functions and
convergence of functionals. \textit{Proc. Amer. Math. Soc.} \textbf{88}, no. 3, 486-490 (1983).

\bibitem{C} Cerami, G.: Un criterio di esistenza per i punti critici su varietà illimitate. \emph{Rc. Ist. Lomb. Sci. Lett.} \textbf{112}, 332-336	(1978).

\bibitem{CZ} Chen, S. and Zhang, D.: Existence of nontrivial solutions for asymptotically linear periodic
Schr\"odinger equations. \textit{Complex Variables and Elliptic Equations: An International Journal} \textbf{60}, 252-267 (2015).

\bibitem{CM} Costa, D. G. and Magalhães, C. A.: A Unified Approach to a Class of Strongly Indefinite Functionals. \emph{Journal of Differential Equations}, \textbf{125}, 521-547 (1996).

\bibitem{CT} Costa, D. G. and Tehrani, H.: Existence and Multiplicity Results for a Class of Schr\"odinger Equations with Indefinite Nonlinearities. \textit{Adv. in Differential Equations} \textbf{8}, 1319-1340 (2003).

\bibitem{CZR} Coti-Zelati Sissa, V. and Rabinowitz, P. H.: Homoclinic Type Solutions
for a Semilinear Elliptic PDE on $\mathbb{R}^N$. \emph{Communications on Pure and Aplied Mathematics}, \textbf{45}, 1217-1269 (1992) 

\bibitem{DJ} Ding, Y. and Jeanjean, L.: Homoclinic Orbits for a non Periodic Hamiltonian System. \emph{Journal of Differential Equations}, \textbf{237} 473-490 (2007) 


\bibitem{DR} Ding, Y. and Ruf, B.: Solutions of a Nonlinear Dirac Equation with External Fields. \textit{Arch. Rational Mech. Anal.} \textbf{190}, 57-82 (2008).


\bibitem{ES} Edelson, A. L. and Stuart, C. A.: The Principle Branch of Solutions of a Nonlinear Elliptic Eigenvalue Problem on $\mathbb{R}^N$. \textit{Journal of Differential Equations} \textbf{124}, 279-301 (1996).

\bibitem{EK} Egorov, Y. and Kondratiev, V.: \textit{On Spectral Theory of Elliptic Operators} vol 89, Birkh\"auser Verlag (1996).


\bibitem{J} Jeanjean, L.: On the Existence of Bounded Palais-Smale Sequences and Application to a
Landesman-Lazer Type Problem Set on $\mathbb{R}^N$. \emph{Proc. Roy. Soc. Edinburgh} \textbf{129A}, 787–809 (1999).

\bibitem{JT} Jeanjean, L. and Tanaka, K.: A Positive Solution for an Asymptotically Linear Elliptic Problem on
$\mathbb{R}^N$ Autonomous at Infinity. \textit{ESAIM: Cont. Opt. Calc. Var.} \textbf{7}, 597-614 (2002).



\bibitem{K} Krasnoselski, M. A.: \textit{Topological Methods in the Theory of Nonlinear Integral Equations}, New York, Macmillan (1964).


\bibitem{KS} Kryszewski, W. and Szulkin, A.: Generalized Linking Theorem with an Application to
Semilinear Schr\"odinger Equation. \textit{Adv. Differ. Equ.} \textbf{3}, 441-472 (1998).

\bibitem{LS} Li, G. and Szulkin A.: An Asymptotically Periodic Schr\"odinger Equation with Indefinite Linear Part. \textit{Communications in Contemporary Mathematics} \textbf{4} No. 4, 763-776 (2002).

\bibitem{LW} Li, G. and Wang, C.: The Existence of a Nontrivial Solution to a Nonlinear Elliptic Problem of  Linking Type without the Ambrosetti-Rabinowitz Condition. \textit{Ann. Acad. Sci. Fenn. Math.} \textbf{36}, 461-480 (2011).


\bibitem{MOR} Maia, L. de A., Oliveira Junior, J. C. and Ruviaro, R.: A Non-periodic and Asymptotically Linear Indefinite Variational Problem in $\mathbb{R}^N$. \textit{Indiana University Mathematics Journal} \textbf{66}No. 1, 31–54 (2017).


\bibitem{P} Pankov, A. A.: \textit{Lecture Notes on Schr\"odinger Equations}, Nova Science Publishers (2007).

\bibitem{Pa} Pankov, A. A.: Periodic Nonlinear Schr\"odinger Equation with Application to
Photonic Crystals. \textit{Milan Journal of Mathematics} \textbf{73}, 259-287 (2005).

\bibitem{R} Rabinowitz, P. H.: \textit{Minimax Methods in Critical Point Theory with Applications to Differential Equations}, American Mathematical Society (1984).

\bibitem{RS4} Reed, M. and Simon, B.: \textit{Methods of Modern Mathematical Physics, Analysis of Operators}, Vol. IV, Academic Press,
New York (1978).

\bibitem{Sc} Schechter, M.: Global Solutions of Nonlinear Schr\"odinger Equations. \textit{Calculus of Variations} \textbf{56:40} (2017).

\bibitem{ScB} Schechter, M.: \emph{Linking Methods in Critical Point Theory}, Birkh\"auser, Boston, (1999).

\bibitem{ScZ2} Schechter, M. and Zou, W: An Infinite-dimensional Linking Theorem and Applications. \textit{Journal of Differential Equations} \textbf{201}, 324-350 (2004).

\bibitem{ScZ1} Schechter, M and  Zou, W.: Weak Linking Theorems and Schr\"odinger with Critical Sobolev Exponent. \textit{ESAIM Control Optim. Calc. Var.} \textbf{9}, 601-619 (2003).


\bibitem{Si} Silva, E. A. B.: Subharmonic Solutions for Sub-quadratic Hamiltonian Systems. \textit{Journal of Differential Equations} \textbf{115} No.1,  120-145 (1995).

\bibitem{So} Soares, M.: \emph{An Abstract Linking Theorem Applied to Indefinite Problems via Spectral Properties}, Ph.D. thesis, University of Brasilia, (2018).

\bibitem{St} Stuart, C. A.: \emph{An Introduction to Elliptic Equations on $\mathbb{R}^N$}, Trieste Notes (1998).

\bibitem{SW}
Szulkin, A. e Weth, T.:
Ground state solutions for some indefinite variational problems. \emph{J. Func. Anal.}
\textbf{257}, 3802-3822 (2009).

\bibitem{SzZo} Szulkin, A. and Zou, W.: Homoclinic Orbits for Asymptitotically linear Hamiltonian Systems. \emph{Journal of Functional Analysis} \textbf{187}, 25–41 (2001).

\bibitem{W} Willem, M.: \textit{Minimax Theorems}, vol 24, Birkh\"auser (1996).
\end{thebibliography}
\end{document}